\documentclass[11pt,reqno]{amsart}
\usepackage{amssymb}
\usepackage[dvips]{epsfig}
\usepackage{graphicx}
\usepackage{color}
\usepackage{amsmath}
\usepackage{amssymb}         
\usepackage{amsfonts}
\usepackage{amsthm}
\usepackage{bbm}

\newcommand{\be}{\begin{equation}}
\newcommand{\ee}{\end{equation}}
\newcommand{\bp}{\begin{proof}}
\newcommand{\ep}{\end{proof}}
\newcommand{\bel}{\begin{equation}\label}
\newcommand{\eeq}{\end{equation}}
\newcommand{\bea}{\begin{eqnarray}}
\newcommand{\eea}{\end{eqnarray}}
\newcommand{\bee}{\begin{eqnarray*}}
\newcommand{\eee}{\end{eqnarray*}}
\newcommand{\ben}{\begin{enumerate}}
\newcommand{\een}{\end{enumerate}}

\newcommand{\R}{\mathbb R}

\newcommand{\Z}{\mathbb Z}

\newcommand{\lop}{\mathcal{L}(\partial_x)}
\newcommand{\pop}{p(\partial_x)}
\newcommand{\japa}{\langle x\rangle}
\newcommand{\sgn}{{\rm sgn}}

\newcommand{\T}{\mathbb{T}}

\newcommand{\sign}{{\rm sgn}}

\newcommand{\ji}{\langle}
\newcommand{\jd}{\rangle}

\newtheorem{theorem}{Theorem}[section]
\newtheorem*{theorem-a}{Theorem A}
\newtheorem*{theorem-b}{Theorem B}
\newtheorem{lemma}[theorem]{Lemma}

\newtheorem{corollary}[theorem]{Corollary}
\theoremstyle{remark}
\newtheorem{remark}{Remark}[section]
\theoremstyle{definition}

\numberwithin{equation}{section}

\title[On properties of solutions to the ILW equation]{On decay and asymptotic properties of solutions to the Intermediate Long Wave equation}
\author{Felipe Linares}
\address[F. Linares]{IMPA\\
Instituto Matem\'atica Pura e Aplicada\\
Estrada Dona Castorina 110\\
22460-320, Rio de Janeiro, RJ\\Brazil}
\email{linares@impa.br}

\author{Gustavo Ponce}
\address[G. Ponce]{Department  of Mathematics\\
University of California\\
Santa Barbara, CA 93106\\
USA.}
\email{ponce@math.ucsb.edu}


\begin{document}


\begin{abstract} 
We consider solutions to the initial value problem associated to the intermediate long wave (ILW) equation. We establish persistence properties of the solution flow in weighted Sobolev spaces, and  show that they  are sharp. We also deal with the long time 
dynamics of large solutions to the ILW equation. Using virial techniques, we describe regions of space 
where the energy of the  solution must decay to zero along sequences of times. Moreover, in the case of exterior regions, we 
prove complete decay for any sequence of times. The remaining regions not treated here are essentially the strong dispersion and soliton regions.
\end{abstract}

\maketitle

\section{Introduction}

This work is concerned with  the Intermediate Long Wave (ILW) equation 
\be\label{ILW} 
\begin{aligned}
 \partial_t u + \mathcal T_\delta \partial_x^2  u + \frac1{\delta}\partial_x u  + u\partial_x u = 0,\qquad (x,t) \in & ~ \R\times \R.
\end{aligned}
\ee
where $u=u(t,x)$ is a real-valued function, 
\be\label{T}
\mathcal T_\delta f(x) := -\, \frac1{2\delta} \,\hbox{p.v.} \int \coth\left(\frac{\pi(x-y)}{2\delta}\right)f(y)dy,
\ee
and $\delta>0$ is a parameter.

Thus, $\mathcal T_\delta$ is a Fourier multiplier of order zero, in the sense that $\partial_x\mathcal T_\delta$ is the multiplier with symbol
\be\label{1.3} 
\sigma(\partial_x\mathcal T_\delta)=\widehat{\partial_x\mathcal T_\delta} (\xi)= - 2\pi \xi \,\hbox{coth}\,(2\pi \delta \xi),
\ee
so
\be\label{1.3b} 
\sigma\Big(\partial_x^2 \mathcal T_{\delta}+\frac{\partial_x}{\delta}\Big)=-i\Big((2\pi\xi)^2\coth(2\pi\delta \xi)-\frac{2\pi\xi}{\delta}\Big):=-i \Omega_{\delta}(2\pi\xi).
\ee

The  ILW equation \eqref{ILW} describes  long internal gravity waves in a two layer stratified fluid, the lower layer having a large finite depth represented by the parameter $\,\delta$, see 
 \cite{Jo}, \cite{KKD}, \cite{JE}, \cite{CL}, \cite{SAK},  \cite{BLS}, and  \cite{CGK} for formal and rigorous derivations of the model.
 
 \medskip
 
In \cite{ABFS}, it was proven that solutions $u(x,t)=u_{\delta}(x,t)$ of the ILW  equation converge, as $\delta \to \infty$ (infinite depth limit), to solutions of the Benjamin-Ono (BO) equation, see \cite{Be}, \cite{On},
 \be\label{BO} 
\partial_t u + \mathcal H\partial_x^2  u   + u\partial_x u = 0,
\ee
with the same initial data. In \eqref{BO}, $\mathcal H$ stands for the Hilbert transform
\begin{equation}
\label{hita}
\begin{aligned}
\mathcal H f(x) :=&~ \frac{1}{\pi} {\rm p.v.}\Big(\frac{1}{x}\ast f\Big)(x)
\\
:=& ~ \frac{1}{\pi}\lim_{\epsilon\downarrow 0}\int\limits_{|y|\ge \epsilon} \frac{f(x-y)}{y}\,dy=(-i\,\sgn(\xi) \widehat{f}(\xi))^{\vee}(x).
\end{aligned}
\end{equation}

Also, in \cite{ABFS} it was shown that  if $u_{\delta}(x,t)$ denotes the solution of the ILW equation \eqref{ILW}, then
\be
\label{scaleKdV}
v_{\delta}(x,t)=\,\frac{3}{\delta} \,u_{\delta}\left(x,\frac{3}{\delta} t\right)
\ee
converges  as $\delta\to 0$ (shallow-water limit) to the solution  of the Korteweg-de Vries (KdV) equation \cite{KdV}
\be\label{KdV} 
\partial_t u + \partial_x^3  u  + u\partial_x u = 0,
\ee
with the same initial data. For recent results in this regard see also \cite{MoVe} and \cite{MPV}.

\medskip

The ILW equation is known to be completely integrable.  The Inverse Scattering formalism was considered in \cite{KAS}, \cite{KSA}, where one finds the Lax pair of the ILW equation, but  no rigorous theory for solving the Cauchy problem by this method is known (see   \cite{KPW} for recent progress on the direct scattering problem). For further comments on general properties of the ILW equation we refer to \cite{JE}, \cite{KAS}, \cite{KKD}, \cite{SAF}, \cite{KPV19} and a recent survey \cite{Sa}.

\medskip
The ILW equation have been also  obtained as a one-dimensional,  uni-directional reduction of a class of ILW {\it systems} derived in \cite{BLS}, \cite{CGK}, see also  \cite{X}.

\medskip
Formally,  real solutions of \eqref{ILW} satisfy  infinitely many conservation laws, due to its integrability, see  \cite{LR}, \cite{Mat3}.  The first three are the following  ones:
\begin{equation}
\label{CL}
\begin{aligned}
I_1(u)&:=\int u(x,t)dx=I_1(u_0),\quad \;\;\;I_2(u):=\int u^2(x,t)dx=I_2(u_0),\\
I_3(u)&:=\int\left( u\mathcal T_\delta \partial_x u +\frac1{\delta} u^2 -\frac13u^3\right)(x,t)dx=I_3(u_0)\\
&~ \qquad  \text{(conservation of the Hamiltonian)}.
\end{aligned}
\end{equation}

The fourth  invariant controls the $H^1$-norm of the solution and is given by

\begin{equation}\begin{split} 
I_4(u)=   &\int_{-\infty}^\infty \Bigg( \frac{1}{4}u^4+\frac{3}{2}u^2\mathcal T_\delta(u_x)+\frac{1}{2}u_x^2+\frac{3}{2}[\mathcal T_\delta(u_x)]^2\\
&\qquad \quad +\frac{1}{\delta} \left( \frac{3}{2}u^3+\frac{9}{2}u\mathcal T_\delta(u_x) \right) +\frac{3}{2\delta^2} u^2\Bigg)dx.
\end{split}
\end{equation}

In general, the $I_k$ invariant, with $k\geq 2$, controls the $H^{(k-2)/2}$-norm of the solution. This is also the case for the BO equation, see \cite{Sa}. For the general structure of the invariants see, for instance, \cite{ABFS}.

It is known that  \eqref{ILW} possesses  soliton (or solitary wave) solutions of the form \cite{Jo}
\be\label{Soliton}
u(x,t)= Q_{\delta,c}(x-ct),  \quad c>\frac1{\delta},
\ee
where $Q_{\delta,c}$  solves
\begin{equation}
\label{solitons}
 \partial_x\mathcal T_\delta  Q_{\delta,c}  +\left( \frac1{\delta} -c\right) Q_{\delta,c}  + \frac12 Q_{\delta,c}^2=0.
\end{equation}
In particular,  $Q_{\delta,c}$ is exponentially localized in space. For the explicit form of $Q_{\delta,c}$, uniqueness, and its stability we refer to 
\cite{A}, \cite{AT}.

The  well-posedness of the initial value problem (IVP) associated to the ILW equation \eqref{ILW} has been studied in 
several works. Using the classical Sobolev spaces $H^s(\R)=(1-\partial^2_x)^{-s/2}L^2(\R)$, it was shown in \cite{ABFS}
 that this IVP is globally well posed for index $s>3/2$. In \cite{MoVe}, this result was extended to $s\geq 1/2$, and latter
  in \cite{MPV} to $s>1/4$, see also \cite{BuPl}.  More recently, there has been a great deal of activity regarding well-posedness in  low regularity spaces. In \cite{IfSa} it was established that the IVP  associated to the ILW equation
\eqref{ILW} is globally well posed in $H^s(\R)$ with $s\geq 0$. In addition, in \cite{IfSa} for small localized data some dispersive estimate were obtained.  Using a different approach in \cite{CLOP} local and global well-posedness was proved for the IVP associated to the ILW in $H^s(\R)$ as well as in $H^s(\T)$, $s\ge 0$. In particular,  the convergence of the ILW dynamics to the BO dynamics in the deep-water limit was proved at the $L^2$-level. For additional well/ill-posedness issues for the IVP associated to \eqref{ILW} see \cite{CFLOP}. 

In the rest of this work, the parameter $\delta$ will be a fixed  strictly positive value. In order to simplify the exposition, in  the proofs of our results we shall take $\delta=1$.

Our first result addresses the persistence properties of solutions of the IVP associated to \eqref{ILW} in polynomial weighted Sobolev spaces:

\begin{theorem}\label{decay}
Let $s\ge \frac12$. Let $u\in C([-T, T] : H^s(\R))$ be a solution to the IVP associated to \eqref{ILW}
obtained in \cite{IfSa}. If $\;|x|^{\alpha}u_0\in L^2(\R)$ with $0\le \alpha\le s$, then
\begin{equation}\label{decay-1}
|x|^{\alpha}u\in C([-T, T] : L^2(\R)).
\end{equation}
\end{theorem}

Roughly, Theorem \ref{decay} tells us that the $L^2$-decay of the data is preserved if it has at least the same amount of regularity in $L^2$. Also, it shows that the solution flow of the ILW preserves the Schwartz class $\mathcal S(\R)$ as does the solution flow of the KdV equation.

The following theorems show that for large values of the parameter, measuring the $L^2$-decay, the result in Theorem \ref{decay} is optimal :

\begin{theorem}\label{main-weights-small} Let $u\in C([-T, T] : H^{1/6}(\R))$ be a solution to the IVP associated to \eqref{ILW}.
If there exist $t_1, t_2\in [-T,T]$, $t_1<t_2$, such that for $\alpha\in[1/6,1/2]$
\begin{equation}\label{decay-hyp1}
|x|^{\alpha}u(x,t_j)\in L^2(\R), \hskip5pt j=1,2,
\end{equation}
then
\begin{equation}\label{regularity1}
u\in C([-T,T] : H^{\alpha}(\R)).
\end{equation}
\end{theorem}

\begin{theorem}\label{main-weights} Let $s>3/2$. Let $u\in C([-T, T] : H^s(\R))$ be a solution to the IVP associated to \eqref{ILW}. 
If there exist $t_1, t_2\in [-T,T]$, $t_1<t_2$, such that for some $\alpha >s$ 
\begin{equation}\label{decay-hyp2}
|x|^{\alpha}u(x,t_j)\in L^2(\R), \hskip5pt j=1,2,
\end{equation}
then
\begin{equation}\label{regularity2}
u\in C([-T,T] : H^{\alpha}(\R)).
\end{equation}
\end{theorem}

\begin{remark}\label{rem-main}\hskip2cm

\begin{enumerate}

\item In the above theorems, the time interval $[-T,T]$ can be taken to be arbitrarily large, see \cite{IfSa}.

\item In  these theorems the relationship between the smoothness, $s$, and the decay, $\alpha$, is dictated by the commutative relation among the operators $\Gamma=x-t\Omega'(\partial_x)$ and the linear part of the ILW equation 
$\partial_t-i\Omega_{\delta}(\partial_x)$, see \eqref{1.3b}.

\item The gap between the result in Theorem \ref{main-weights-small} and those in Theorem \ref{main-weights}, i.e. $s\in(1/2,3/2]$, is due to the arguments used in their proofs. These are mainly based on weighted energy estimates. Thus, when $\alpha\leq 1/2$ the structure of the non-linear term is preserved with the weight, so it can handle it by integration by parts. For $\alpha>1/2$ this is not the case.  One needs to rely on an iterative argument which at each step increases in  $1/2$  the desired regularity. This is coming from the form of the local smoothing effect. This provides a gain of a $1/2$-derivative (described by a nonlocal operator). To establish this smoothing effect,  one has to estimate terms involving  fractional derivatives and fractional weights. In our arguments, it is crucial to have the quantity $\underset{t}{\sup} \| \partial_xu(t)\|_{\infty}$ controlled  \it a priori \rm. This explains the hypothesis $s>3/2$ in Theorem \ref{main-weights}.

\item  For the $k$-generalized KdV
\begin{equation*}
\partial_t u + \partial_x^3  u  + u^k\partial_x u = 0,\;\;\;\;\;\;k=1,2,.....
\end{equation*}
similar results to those in Theorem \ref{decay} for $\alpha\in \mathbb N$ were proven in \cite{Ka}, and for $\alpha\in \mathbb R$ in \cite{NP2} and \cite{FLP15}. Theorems \ref{main-weights-small}-\ref{main-weights} were deduced in \cite{FLP15} and \cite{NP2}. In this case, the local smoothing effect gives a gain of one classical derivative. Moreover, the local well-posedness is based in the contraction principle which can be extended to obtain the above theorems for this model.

\item As in the case for the BO equation, see \cite{MoSaTz},  one should not expect a well-posedness result based on  solely the contraction principle.

\item For the BO equation, similar results to those in Theorems \ref{decay}-\ref{main-weights} do not hold, see \cite{Io}, \cite{FP}, and \cite{FLP13}. For example, Theorem \ref{decay} only holds for $\alpha\leq 7/2$, and this value is sharp, see \cite{FP}. This is mainly due to the non-smoothness of the regularity of the operator modeling the dispersive effect.

\item From the  well-posedness result obtained in \cite{IfSa}  to obtain Theorem \ref{main-weights-small} and Theorem \ref{main-weights} it suffices to show that there exists 
$\hat{t}\in (-T,T)$ such that
$u(\cdot,\hat{t})\in H^{\alpha}(\R)$.
\end{enumerate}
\end{remark}

\medskip

Next, we shall study the long time behavior of the global solutions of the ILW equation \eqref{ILW}. This study was initiated in \cite{MPS}, motivated by previous works \cite{MuPo1} and \cite{MuPo2} for the KdV and the BO  equation respectively. Also, here we shall follow some  of the ideas introduced in 
\cite{MMPP} which improve some of the results in  \cite{MuPo1} and \cite{MuPo2}, see also \cite{FLMP}.

\medskip

\begin{theorem}\label{L2ILW}
Let $u_0 \in L^2(\mathbb{R}) $ and  $u = u(x,t)$ be the global in time solution of the IVP associated to the \eqref{ILW}  such that 
\begin{equation*}
u \in C(\mathbb{R}:L^2(\mathbb{R}))\cap L^{\infty}(\mathbb{R}: L^2(\mathbb{R})).
\end{equation*}
Then 
  \begin{equation}\label{EQ1}
   \liminf_{t \to \infty }\int_{B_{t^b}(0)}u^2(x,t)\, dx = 0,
  \end{equation}
 where $B_{t^b}(0)$ denotes the ball centered in the origin with radius $t^b$,
 \begin{equation}
  B_{t^b} (0):= \{x \in \mathbb{R} : |x| < t^b\} \quad \mbox{with}\quad 0 < b < \frac{2}{3}.
 \end{equation} 
 
Moreover, there exist a constant $C > 0$ and an increasing sequence of times $t_n \to \infty$, such that 
    \begin{equation}\label{EQ2}
    \int_{B_{t_n^b}(0)}u^2(x,t_n)\, dx \leq \frac{C}{\log^{\frac{(1-b)}{b}}(t_n)}.
    \end{equation}

\end{theorem}

As a consequence of Theorem \ref{L2ILW} we obtain:
\medskip

\begin{corollary}\label{corL2ILW} Let $u_0 \in L^2(\mathbb{R}) $ and  $u = u(x,t)$ be the global in time solution to the
IVP associated to the \eqref{ILW}   such that 
$$u \in C(\mathbb{R}:L^2(\mathbb{R}))\cap L^{\infty}(\mathbb{R}: L^2(\mathbb{R})).$$
Then 
 \begin{equation}\label{notcentered}
   \liminf_{t \to \infty }\int_{B_{t^b}(t^m)}u^2(x,t)\, dx = 0,
  \end{equation}
  where
  \begin{equation}
 B_{t^b}(t^m) := \{ x \in \mathbb{R} :|x - t^m|< t^b\},
\end{equation}
with
\begin{equation}\label{conditions-on-m}
0 < b < \frac{2}{3} {\hskip15pt\text{and} \hskip15pt}    0  \leq m < 1 - \frac{3}{2}b.
\end{equation}
\end{corollary}

\medskip

Next, we present a result concerning the decay of solutions in the energy space :

\begin{theorem}\label{energy-ILW}
Let $u_0 \in H^{1/2}(\mathbb{R}) $ and $u = u(x,t)$ be the global in time solution of the  IVP associated to \eqref{ILW} such that 
\begin{equation*}
 u \in C(\mathbb{R}:H^{1/2}(\mathbb{R}))\cap L^{\infty}(\mathbb{R}: H^{1/2}(\mathbb{R})).
\end{equation*}
   Then 
  \begin{equation}\label{EQ3}
   \liminf_{t \to \infty }\int_{B_{t^b}(0)}\left(u^2(x,t)+|q(\partial_x)u(x,t)|^2\right)\, dx = 0,\;\;\;\;0<b<\frac{2}{3}.
  \end{equation}
  where the operator $q(\partial_x)$ is defined below (see \eqref{a7}).
\end{theorem}

\medskip

Now, we will consider the asymptotic decay of the solution in a domain moving in time in the right direction: 

\begin{theorem}
\label{to-the-right}
There exists a constant $C_0>0$ depending only on $\|u_0\|_{H^1}$ such that the global  solution 
\[
u\in C(\mathbb R:H^1(\mathbb R))\cap L^{\infty}(\mathbb R:H^1(\mathbb R))
\]
of IVP associated to \eqref{ILW} satisfies
\begin{equation}
\label{main}
\lim_{t\to \infty} \| u(t)\|_{L^2(x\geq C_0t)}=0.
\end{equation}
\end{theorem}

\medskip

\begin{remark}\hskip10pt

\begin{enumerate}
\item \label{to-the-left} In \cite{MPS} it was studied the decay of the $L^2$-norm of the solution in the far left region.
It was proved that, despite the size of the data, linear waves have small influence in the region 
$\{-t\log^{1+\epsilon}t\ll x\le 0\}$ for $\epsilon>0$. See Theorem 1.6 in \cite{MPS} for the details.

\item Collecting the information in Theorems \ref{L2ILW}, \ref{to-the-right} and the latter remark one can deduce several estimates. In particular, one has: 
there exists $\,C_0=C_0(\|u_0\|_{H^1})>0$ and an increasing sequence of times $(t_n)_{n=1}^{\infty}$ with $t_n\uparrow \infty$ as $n\to \infty$ such that for any 
constants $c>0,\,\gamma>0$,
\begin{equation}
\label{123}
\liminf_{n\to \infty} \,\int_{\Omega(t_n)} |u(x,t_n)|^2\,dx=\|u_0\|_2^2,
\end{equation}
with
\[
\Omega(t): =\{x\in\R : -c\,t\,\log^{1+\gamma}t<x<-c\,t^{\frac23^{-}}\hskip5pt \text{or}\hskip7pt c\,t^{\frac23^{-}}\!\!\!<x<C_0\,t\}.
\]
\end{enumerate}
\end{remark}


The rest of this paper is organised as follows: Some preliminary results useful in our analysis will be presented in
Section 2. The proof of decay properties of solutions to the ILW equation in Theorem \ref{decay} is contained
in Section 3.  In Sections 4 and 5 we give the proofs of Theorem \ref{main-weights-small} and Theorem  \ref{main-weights}
which show the sharpness of the results in Theorem  \ref{decay}. Sections 6 contains the proof of the result concerning the
asymptotic decay of the solution in a domain moving in time in the right direction.  The asymptotic behavior of the solution
for any data in $L^2(\R)$ is given in Section 7. Finally, in Section 8 we sketch the proofs of the asymptotic dynamics 
of solutions to the ILW equation in the energy space and the noncentered $L^2(\R)$ case.

\section{Preliminary Results}

This section contains several estimates to be used in the proof of the results stated inn the introduction.

\subsection*{Classical $\Psi$.d.o's in $S^{^{m}}_{1,0}$}

The symbol class
\begin{equation}
\begin{split}
S^{^{m}}_{1,0}=\Big\{ & a\in C^{\infty}(\R^{2n})\,: \; \forall \alpha,\beta\in (\Z^{+})^n\\
&\hskip15pt\big|\partial_x^{\alpha}\partial_{\xi}^\beta a(x,\xi)\big| \le c_{\alpha,\beta} \big(1+|\xi|\big)^{^{m-|\beta|}}
\Big\}.
\end{split}
\end{equation}

\begin{lemma}\label{pseudo+weights}
 If $a\in S^{^{0}}_{1,0}$. Then for any $p\in(1,\infty)$ and for any $\rho\in \R$
 
 \begin{equation}\label{lem1-pseudo}
 a(x,D): L^p (\R^n: \,\ji x\jd^{\rho} \,dx) \to L^p (\R^n : \,\ji x\jd^{\rho} \,dx)
 \end{equation}
 with $\ji x\jd =(1+|x|^2)^{\frac12}$.
 \end{lemma}
 
 \medskip
 
 \begin{remark}
\hskip10pt
 \begin{enumerate}
 \item
 The case $p=2$ and $\rho>0$ was proven in \cite{NP} (see also \cite{KPV98}).
 
 \item We recall that if
 $$
 T_af(x)= a(x,D)f(x)= c\int e^{ix\xi} a(x,\xi) \widehat{f}(\xi)\,d\xi,
 $$
 with $a\in S^{^{0}}_{1,0}$, then for any $p\in(1,\infty)$ there exists $c_p>0$ such that
 $$
 \| T_af\|_p\leq c_p \,\|f\|_p,
 $$
see \cite{St}.
\end{enumerate}
 
 \end{remark}

 \begin{proof}  To simplify the exposition we fix $n=1$. It suffices to prove it for $\rho=2pl$ with $l\in\Z$.  
 Since 
 $$
 a(x,D)f(x)= c\int e^{ix\xi} a(x,\xi) \widehat{f}(\xi)\,d\xi,
 $$
 to prove \eqref{lem1-pseudo} we need to show that
 
\begin{equation*}
 \int |a(x,D) f(x)|^p \ji x\jd^\rho\,dx \le c_{\rho}\int |f(x)|^p \ji x\jd^\rho\,dx
 \end{equation*}
 with $\rho=2pl$, $l\in\Z$.
 
 Thus, if $f(x) \ji x\jd^{2l}=g(x)$, we need to show 
 \begin{equation*}
 \int \Big| a(x,D)\Big(\frac{g(x)}{\ji x\jd^{2l}}\Big)\Big|^p \ji x\jd^{2pl}\,dx \le c_{\rho}\int |g(x)|^p \,dx,
 \end{equation*}
 i.e.
 \begin{equation*}
\big\| \int e^{ix\xi} a(x,\xi) \Big(\widehat{\frac{g(x)}{\ji x\jd^{2l}}}\Big)(\xi) \ji x\jd^{2l}\,d\xi\big\|_{p}^p\le c_{\rho} \|g\|_{p}^p.
\end{equation*} 

\medskip

\noindent{\underline{\bf Case 1:}} $l>0$.

\medskip

Observe that
\begin{equation*}
\ji x\jd^{2l} e^{ix\xi}= \underset{j=0}{\overset{l}{\sum}} c_{jl}\,\partial_{\xi}^{2j} \big(e^{ix\xi}\big)\hskip10pt \text{with\hskip5pt} c_{jl}=(-1)^j
\begin{pmatrix}
l\\
j
\end{pmatrix}.
\end{equation*}

Hence

\begin{equation*}
\begin{split}
&\int \ji x\jd^{2l} e^{ix\xi} a(x,\xi) \Big(\widehat{\frac{g(x)}{\ji x\jd^{2l}}}\Big)(\xi)\,d\xi\\
&= \underset{j=0}{\overset{l}{\sum}} c_{jl}\,\int \partial_{\xi}^{2j} \big(e^{ix\xi}\big)a(x,\xi) \Big(\widehat{\frac{g(x)}{\ji x\jd^{2l}}}\Big)(\xi)\,d\xi\\
&= \underset{j=0}{\overset{l}{\sum}} c_{jl}\,\int e^{ix\xi}\, \partial_{\xi}^{2j}\Big(a(x,\xi) \Big(\widehat{\frac{g(x)}{\ji x\jd^{2l}}}\Big)(\xi)\Big)\,d\xi
\end{split}
\end{equation*}
with
\begin{equation*}
\begin{split}
&\partial_{\xi}^{2j}\Big(a(x,\xi) \Big(\widehat{\frac{g(x)}{\ji x\jd^{2l}}}\Big)(\xi)\Big)\\
&=\underset{m=0}{\overset{2j}{\sum}}
\begin{pmatrix}
2j\\
m
\end{pmatrix}
\partial_{\xi}^{2j-m} a(x,\xi) \partial_{\xi}^{m}\Big(\widehat{\frac{g(x)}{\ji x\jd^{2l}}}\Big)(\xi)\\
&=\underset{m=0}{\overset{2j}{\sum}}
\begin{pmatrix}
2j\\
m
\end{pmatrix}
c_m^j\partial_{\xi}^{2j-m} a(x,\xi) \Big(\widehat{\frac{x^m g(x)}{\ji x\jd^{2l}}}\Big)(\xi)
\end{split}
\end{equation*}
with $m\le 2l$.

\medskip

Collecting the information above, using that $\partial_{\xi}^r a(x,D) \in S^{^0}_{1,0}$ and
\begin{equation*}
\Big\|\widehat{\frac{x^m g(x)}{\ji x\jd^{2l}}}\Big\|_p\le \|g\|_p
\end{equation*}
we obtain the case $l>0$.

\medskip

\noindent{\underline{\bf Case 2:}} $l<0$.

\medskip

We have that
\begin{equation*}
\begin{split}
&\int e^{ix\xi} a(x,\xi)\, \widehat{\ji x\jd^{2l} g}(\xi)\,\frac{d\xi}{\ji x\jd^{2l}}\\
&= \int \frac{e^{ix\xi} a(x,\xi)}{\ji x\jd^{2l}} (1-\Delta_{\xi})^{2l} \,\widehat{g}(\xi)\,d\xi\\
&= \int (1-\Delta_{\xi})^{2l}\Big(\frac{e^{ix\xi} a(x,\xi)}{\ji x\jd^{2l}}\Big)\, \widehat{g}(\xi)\,d\xi.
\end{split}
\end{equation*}

Since
\begin{equation*}
(1-\Delta_{\xi})^{2l}\Big(\frac{e^{ix\xi} a(x,\xi)}{\ji x\jd^{2l}}\Big) = e^{ix\xi} a_l (x,\xi)
\end{equation*}
with $a_l\in S^{^0}_{1,0}$, the result follows.
  
\end{proof}

\medskip

Next, we present some new interpolation estimates useful for our analysis which are interesting in
their own right. The main new ingredient in the proof is given by the result in Lemma \ref{pseudo+weights}.

\begin{lemma}\label{lem2-interpolation} Let $p\in (1,\infty)$. Let $a, b\in\R$. Then for any $\theta\in (0,1)$,
\begin{equation}\label{interpol-1}
\| \ji x\jd^{\theta b}J^{(1-\theta)a} f\|_p \le c_p\|\ji x\jd^{b}f\|_p^{\theta}\|J^a f\|_p^{1-\theta},
\end{equation}
with $$
J^af(x) =(1-\Delta)^{a/2}f,\;\;\;\;\;a\in\R.
$$
\end{lemma}

\begin{remark} The case $p=2$ and $a, b>0$ was proven in \cite{NP}.
\end{remark}

\begin{proof} It follows the argument in \cite{NP}, which is based in the Three Lines  Theorem. Define
\begin{equation*}
F(z)=\int e^{(z^2-1)} \ji x\jd^{bz} J^{a(1-z)} f(x) \overline{g(x)}\,dx
\end{equation*}
with $g\in L^q(\R^n)$, $\|g\|_q=1$ with $1/p+1/q=1$.

Let $D=\{z=\eta+iy \; | \; 0\le \eta \le 1 \}$.

$F$ is continuous in $D$ and analytic in its interior with
\begin{equation}\label{interpol-2}
|F(0+iy)|\le c e^{-(y^2+1)} (1+(ay)^2)^k \|J^af\|_p
\end{equation}
for some $k=k(n,p)\in \Z^{+}$ (using that $J^{iay}$ is a $\Psi$.d.o of order zero and Lemma \ref{pseudo+weights})

On the other hand,
\begin{equation}\label{interpol-3}
\begin{split}
|F(1+iy)|&\le c e^{-y^2} \|\ji x\jd^b J^{iay}f\|_p\\
&\le c e^{-y^2}(1+(ay)^2)^k \|\ji x\jd^b f\|_p
\end{split}
\end{equation}
by the previous lemma.

Combining \eqref{interpol-2} and \eqref{interpol-3} we get the desired result.
\end{proof}

Next, we present an extension of the preceding interpolation result.

\begin{lemma}\label{lem2b-interpolation} Let $p\in (1,\infty)$. Let $a,b,c,d\in\R$. Then for $\theta\in[0,1]$ it holds that
\begin{equation}\label{interpol-complete}
\|\japa^{\theta a+(1-\theta)c}J^{\theta b+(1-\theta)d} f\|_p \le c_p\|\japa^{a}J^{b} f\|_p^{\theta}\|\japa^{c}J^{d} f\|_p^{(1-\theta)}.
\end{equation}
\end{lemma}

\begin{proof} Let $g\in L^q(\R)$, $\|g\|_q=1$ with $1/p+1/q=1$. Consider
\begin{equation*}
F(z)=\int e^{(z^2-1)}\japa^{za+(1-z)c} J^{zb+(1-z)d} f\cdot g\,  dx.
\end{equation*}

$F$ is continuous on $\Omega=\{z=\eta+i y\;:\; 0\le\eta\le1\}$ and analytic in the interior of $\Omega$, i.e.
 $\overset{\circ}{\Omega}$.

For some $k= k(n)$, we have that

\begin{equation*}
|F(0+iy)|\le e^{-(y^2+1)} (1+(by)^2+(dy)^2)^k\|\japa^c J^df\|_p
\end{equation*}
and 
\begin{equation*}
|F(1+iy)|\le c \,e^{-(y^2+1)} (1+(by)^2+(dy)^2)^k\|\japa^a J^bf\|_p
\end{equation*}
where we have used that $J^{iy}$ is a $\Psi$.d.o of order zero.

From the general form of the three lines Theorem we obtain the desired result \eqref{interpol-complete}.
\end{proof}

\vspace{3mm}

\subsection*{Properties of symbol associated to ILW equation} \hskip10pt

\vspace{3mm}

Using the same notation as in \cite{MPS}, we write
\begin{equation}\label{ilw-oper}
\sigma\Big(\mathcal T_{\delta}\partial_x^2+\frac{\partial_x}{\delta}\Big)(\xi)=-i\Big((2\pi\xi)^2\coth(2\pi\delta \xi)-\frac{2\pi\xi}{\delta}\Big):=-i \Omega_{\delta}(2\pi\xi).
\end{equation}

Fix $\delta\equiv 1$ to simplify the exposition. Thus
\begin{equation}\label{omega1}
\Omega_1(\xi)=\begin{cases}
\dfrac{\xi^3}{3}+O(\xi^5), \hskip40pt |\xi|\downarrow 0,\\
\\
\xi|\xi|-\xi+O(1), \hskip20pt |\xi|\uparrow \infty,
\end{cases}
\end{equation}
with
\begin{equation}\label{der-omega1}
\Omega'_1(\xi)=\begin{cases}
\xi^2+O(\xi^4), \hskip43pt |\xi|\downarrow 0,\\
\\
2|\xi|-1+O(\frac{1}{\xi}), \hskip20pt |\xi|\uparrow \infty,
\end{cases}
\end{equation}
and $\Omega'_1(\cdot)$ even, $\Omega_1'(0)=\Omega_1''(0)=0$.

$$
\Omega'_1(\xi)=\xi^2-\dfrac{\xi^4}{9}+O(|\xi|^6),\hskip10pt |\xi|\downarrow 0.
$$

Hence
\begin{equation*}\label{p-symbol}
p(\xi)=\sqrt{\Omega'_1(\xi)}= \xi\Big(1-\frac{\xi^2}{9}\Big)^{\frac12} + O(|\xi|^3) \hskip5pt\text{as}\hskip5pt |\xi|\downarrow 0.
\end{equation*}

($p(\xi)$ smooth around zero)
$$
p(\xi)\simeq \sqrt{2} \sqrt{|\xi|}\,\sign(\xi) \hskip15pt |\xi|\uparrow \infty
$$
then we can see that $p(\xi)$ is smooth and odd with $\Omega'(\xi)=p(\xi)p(\xi)$.

Next, as in \cite{MPS} we define the operator $\,q(\partial_x)$ whose symbol $\,q(\xi)$ is the square root of  $\Omega_{1}'(\xi)$, i.e.
\be
\label{a7}
\Omega'_{1}(\xi)=q(\xi)\,q(\xi),
\ee
such that $q(\cdot)$ is even, Lipschitz, with $q(\xi)>0$ for $\xi>0$. Hence,
\begin{equation}
\label{a8}
q(\xi)= \xi+O\Big(\xi^3\Big)\;\;\;\;\;\text{as}\;\;\;\;\;\xi\downarrow 0,
\end{equation}
and
\begin{equation}
\label{a9}
q(\xi)=\sqrt{\xi}\Big(1-\frac{1}{2\xi}+O(\xi^{-2})\Big)\;\;\;\;\;\text{as}\;\;\;\;\;\xi\uparrow \infty,
\end{equation}
with
\be
\label{a10}
q(0)=0\;\;\;\;\;\;\text{and}\;\;\;\;\;\;q'(0^+)=1.
\ee

We also use the following notation,
\begin{equation}\label{operatorL}
-\mathcal{L}(\partial_x)\partial_x= \mathcal{T}\partial_x^2+\partial_x,
\end{equation}
see \eqref{ilw-oper}.

\medskip

\subsection*{Commutators} \hskip15pt

\vspace{3mm}

We start this section by establishing a useful commutator estimate which is fundamental in our analysis.

Suppose $\varphi(\cdot)$ is a weight (smooth) with $\varphi(x)\simeq \ji x\jd^m$, for some $\;m>0$. 

Let $Q(D)$ be a $\Psi$.d.o. in $S^l_{1,0}(\R)$ given by
\begin{equation}\label{operator-q}
Q(D)f(x)=c\int e^{ix\xi} Q(\xi) \widehat{f}(\xi)\, d\xi.
\end{equation}

In particular, $Q(\xi)$ is smooth. 

\begin{lemma}\label{lem3}
Consider the operator $Q$ defined in \eqref{operator-q}. Then for $k>l$ and $k\ge m$ it holds that
\begin{equation}\label{lem3-comm}
[Q(\partial_{\xi});\varphi(x)]f(x)= \underset{j=1}{\overset{k}\sum} \varphi^{(j)} Q^{(j)}(\partial_{\xi})f+\mathcal{R}_kf
\end{equation}
where
\begin{equation}\label{lem3-comm-1}
\|\mathcal{R}_kf\|_2\le c \|\varphi^{(k)}\|_{1,2}\|f\|_2.
\end{equation}
\end{lemma}

\begin{proof} We notice first that

\begin{equation}\label{lem3-comm-2}
\begin{split}
&\Big([Q(\partial_{\xi});\varphi(x)]f(x)\Big)^{\widehat{\hskip3pt}}(\xi)\\
&=Q(\xi)\widehat{\varphi f}(\xi)-\widehat{\varphi(x)Q(\partial_{\xi})f}(\xi)\\
&=Q(\xi)(\widehat{\varphi}*\widehat{f})(\xi)-(\widehat{\varphi}*\widehat{Q(\partial_{\xi})f})(\xi)\\
&=\int Q(\xi) \widehat{\varphi}(\xi-\eta)\widehat{f}(\eta)\,d\eta- 
\int Q(\eta)\widehat{\varphi}(\xi-\eta)\widehat{f}(\eta)\,d\eta\\
&=\int \widehat{\varphi}(\xi-\eta)\big(Q(\xi)-Q(\eta)\big) \widehat{f}(\eta)\,d\eta.
\end{split}
\end{equation}

By a Taylor expansion we find that
\begin{equation*}
\begin{split}
Q(\xi)-Q(\eta)&=\underset{j=1}{\overset{k}{\sum}} c_j \,Q^{(j)}(\eta)(\xi-\eta)^j\\
&\hskip15pt+ c_{k+1} Q^{(k+1)}(\theta\xi+(1-\theta)\eta)(\xi-\eta)^{k+1}
\end{split}
\end{equation*}
with $\theta\in [0,1]$, $\theta=\theta(\xi,\eta)$. Inserting the identity above in the last expression
on the right hand side of \eqref{lem3-comm-2} we get
\begin{equation}\label{lem3-comm-3}
\begin{split}
&\Big([Q(\partial_{\xi});\varphi(x)]f(x)\Big)^{\widehat{\hskip3pt}}(\xi)\\
&=\underset{j=1}{\overset{k}{\sum}} c_j \int (\xi-\eta)^j  \widehat{\varphi}(\xi-\eta) Q^{(j)}(\eta)\widehat{f}(\eta)\,d\eta\\
&\hskip15pt+ c_{k+1} \int Q^{(k+1)}(\theta\xi+(1-\theta)\eta)(\xi-\eta)^{k+1}\widehat{\varphi}(\xi-\eta)\widehat{f}(\eta)\,d\eta\\
&=\underset{j=1}{\overset{k}{\sum}} c_j \int \widehat{\partial_x^j\varphi}(\xi-\eta) \widehat{\big(Q^{(j)}(\partial_{\xi})f\big)}(\eta)\,d\eta + \mathcal{R}_{k+1}f\\
&=\underset{j=1}{\overset{k}{\sum}} c_j \widehat{\big(\partial_x^j\varphi \cdot Q^{(j)}(\partial_{\xi})f\big)}(\xi)
+ \mathcal{R}_{k+1}f.
\end{split}
\end{equation}

Thus
\begin{equation}\label{lem3-comm-4}
\begin{split}
[Q(\partial_{\xi});\varphi(x)]f(x)= \underset{j=1}{\overset{m}{\sum}} c_j \partial_x^j\varphi \,Q^{(j)}(\partial_{\xi})f
+ \widehat{\mathcal{R}_{k+1}f}
\end{split}
\end{equation}
where
\begin{equation*}
\mathcal{R}_{k+1}f(x)=c_{k+1} \int Q^{(k+1)}(\theta\xi+(1-\theta)\eta)\,\widehat{\partial_x^{k+1}\varphi}(\xi-\eta)\widehat{f}(\eta)\,d\eta.
\end{equation*}

Suppose $\|Q^{(k+1)}\|_{\infty}\le M$ (it is sufficed to take $k+1\ge l$ because $Q\in S^l_{1,0}$), then
\begin{equation*}
|\mathcal{R}_{k+1}f(x)|\le |c_{k+1}|\, M\,\int|\widehat{\partial_x^{k+1}\varphi}(\xi-\eta)|\,|\widehat{f}(\eta)|\,d\eta.
\end{equation*}
Thus by employing Plancherel identity and Sobolev we have that
\begin{equation*}
\begin{split}
\|\mathcal{R}_{k+1}f\|_2&\le c \|\widehat{f}\|_2\|\widehat{\partial_x^{k+1}\varphi}\|_1\\
&\le c \|\partial_x^{k+1}\varphi\|_{1,2}\|f\|_2.
\end{split}
\end{equation*}

We need $\partial_x^{k+1}\varphi\in L^2$. Since $\varphi \simeq \ji x\jd^m$, it is suffices to take $k\ge m$.
\end{proof}

Next, we present a commutator estimate established in \cite{KaPo}, and an extension given in  \cite{GO}.
\begin{theorem} Let $s>0$, $1<p<\infty$, and  $1/p=1/p_1+1/p_2=1/p_3+1/p_4,$ with $p_2, p_3<\infty$. Then there exists $c=c_{s,n,p,p_1,p_2,p_3,p_4}>0$ such that for all $f, g\in \mathcal{S}(\R^n)$, it holds that
\begin{equation}\label{kp-comm}
\|[J^s, f]g\|_p\le c \big(\|\nabla f\|_{p_1}\|J^{s-1}g\|_{p_2}+\|J^sf\|_{p_3}\|g\|_{p_4}\big).
\end{equation}
\end{theorem}

In the case $p=p_2=p_3$, the proof of \eqref{kp-comm} was given in \cite{KaPo}. The general case in \eqref{kp-comm} follows the same argument using the version of the multiplier result of Coifman-Meyer \cite {CoMe} stated in  \cite{GO} (Theorem A).

\noindent 
\begin{lemma}\label{lem1}
 Let $p\in[1,\infty]$ and $s\geq 0$, with $s$ not an odd integer in the case $p=1$. Then 
\begin{equation}\label{comm-go-1}
\|J^{s} (fg)\|_p \leq c (\|f\|_{p_1}\|J^{s}g\|_{p_2} +\|g\|_{q_1}\|J^{s}f\|_{q_2}),
\end{equation}
and
\begin{equation}\label{comm-go-2}
\|D^{s} (fg)\|_p \leq c (\|f\|_{p_1}\|D^{s}g\|_{p_2} +\|g\|_{q_1}\|D^{s}f\|_{q_2})
\end{equation}
with $1/p_1+1/p_2=1/q_1+1/q_2=1/p$.

\end{lemma}

Lemma \ref{lem1} with $p\in(1,\infty)$ we refer to  \cite{KaPo}, \cite{KPV93} and \cite{GO}. 
The case $r=p_1=p_2=q_1=q_2=\infty$ was established in \cite{BoLi}, see also \cite{GMN}. 
The case $p=p_2=q_2=1$ was settled in \cite{OW}.

\medskip

The next estimate is an inequality of Gagliardo-Nirenberg type whose proof can be found in  \cite{NP}.

\begin{lemma}\label{GNSinequality} There exists $c>0$, such that for any $f\in H^{1/2}(\R)$ 
 \begin{equation}
 \|f\|_{L^3} \leq c\,\|f\|_{L^2}^{\frac{2}{3}}\|D^{1/2}f\|_{L^2}^{\frac{1}{3}}. 
 \end{equation}
\end{lemma}

\vspace{0.5cm}

\section{Decay}

In this section we show persistence properties of the solution flow associated to the ILW equation in weighted Sobolev spaces.

\begin{proof}[Proof of Theorem \ref{decay}]
Formally multiply the equation \eqref{ILW} by $\japa^{2\alpha}u$ and integrate the result in the $x$-variable to get
\begin{equation}\label{decay-star}
\frac12\frac{d}{dt}\int u^2(x,t)\japa^{2\alpha}\,dx+\underbrace{\int{\lop}\partial_xu\,u \japa^{2\alpha}\,dx}_{ A_1(t)}+
\underbrace{\int u\partial_xu \,u \japa^{2\alpha}\,dx}_{ A_2(t)}=0.
\end{equation}

Notice
\begin{equation}\label{decay-2}
|A_2(t)|\le \left|-\frac{2\alpha}{3}\int u^3 \japa^{2\alpha-1}\frac{x}{\japa}\,dx\right| \le c\|\japa^{\alpha}u\|_2\|\japa^{\alpha-1}u^2\|_2.
\end{equation}
Thus, if $\alpha\le1$, Sobolev embedding yields
\begin{equation*}
\|\japa^{\alpha-1}u^2\|_2\le \|u\|_4^2\le c\|u\|_{\frac12,2}^2.
\end{equation*}
Thus, if $\alpha>1$ ($s>1$), using Sobolev embedding it follows that
\begin{equation*}
\|\japa^{\alpha-1}u^2\|_2\le \|\japa^{\alpha-1}u\|_2\|u\|_{\infty}\le c\|\japa^{\alpha}u\|_2\|u\|_{1,2}.
\end{equation*}
inserting this in \eqref{decay-2} we get a bound for $|A_2(t)|$.

Next we estimate the term $A_1(t)$.

\begin{equation}\label{decay-3}
\begin{split}
A_1(t)&= -\int u \lop\partial_x(u\japa^{2\alpha})\,dx\\
&= -\int u\lop\partial_xu \japa^{2\alpha}\,dx -\int u\,[\lop\partial_x; \japa^{2\alpha}]u\,dx.
\end{split}
\end{equation}

Hence
\begin{equation*}
2A_1(t)= -\int u\,[\lop\partial_x; \japa^{2\alpha}]u\,dx.
\end{equation*}

Now,
\begin{equation*}
\begin{split}
[\lop\partial_x; \japa^{2\alpha}]&= c\japa^{2\alpha-1}\frac{x}{\japa}\Omega'(\partial_x)
+\phi_{2\alpha-2}(x)\Omega''(\partial_x)+\dots+\mathcal{R}\\
&=  c\japa^{2\alpha-1}\frac{x}{\japa}\Omega'(\partial_x)+B_1.
\end{split}
\end{equation*}
where 
\begin{equation*}
|\phi_{2\alpha-2}(x)|\le c\japa^{\alpha-1}\hskip10pt \text{and} \hskip10pt \left|\int B_1uu\,dx\right|\le c\|\japa^{\alpha-1}u\|_2^2.
\end{equation*}

Above we have used the commutator expansion (Lemma \ref{lem3}), Lemma \ref{pseudo+weights} with the fact that 
$\Omega^{(j)}(\partial_x)\in S^{0}_{1,0}$ if $j\ge 2$. 

Thus we just need to handle
\begin{equation*}
\begin{split}
&\int \japa^{2\alpha-1}\frac{x}{\japa}\Omega'(\partial_x)u \,u dx=\int \japa^{2\alpha-1}\frac{x}{\japa}\pop\pop u\,u dx\\
&=-\int \left[\pop;\japa^{2\alpha-1}\frac{x}{\japa}\right]\pop u\,u dx
+\int \pop\left(\japa^{2\alpha-1}\frac{x}{\japa}\pop u\right)\, u dx\\
&=-\int \left[\pop;\japa^{2\alpha-1}\frac{x}{\japa}\right]\pop u\,u dx- \int \japa^{2\alpha-1}\frac{x}{\japa}\pop u\,\pop u dx\\
&= A_{11}(t)+A_{12}(t).
\end{split}
\end{equation*}

Using Lemma \ref{pseudo+weights}, Lemma \ref{lem2-interpolation}, and Young's inequality we have the following
chain of inequalities
\begin{equation*}
\begin{split}
|A_{12}(t)|&\le \|\japa^{\alpha-\frac12} \pop u\|_2^2\le \|\japa^{\alpha-\frac12}J^{\frac12} u\|_2^2\\
&\le c(\|\japa^{\alpha}u\|_2^{1-\frac{1}{2\alpha}}\|J^{\alpha}u\|_2^{\frac{1}{2\alpha}})^2\\
&\le c\|\japa^{\alpha}u\|_2^2+\|u\|_{s,2}^2.
\end{split}
\end{equation*}

To control $A_{11}$  we write by using the commutator expansion (Lemma \ref{lem3})
\begin{equation*}
\begin{split}
\big[\pop;&\japa^{2\alpha-1}\frac{x}{\japa}\big]\pop \\
&= \phi_{2\alpha-2}(x)p'(\partial_x)\pop+ 
\phi_{2\alpha-3}(x)p''(\partial_x)\pop+\dots +\mathcal{R}\\
&=\phi_{2\alpha-2}(x)p'(\partial_x)\pop +B_2,
\end{split}
\end{equation*}
where 
\begin{equation*}
\Big|\int B_2 u u\, dx\Big| \le c\|\japa^{\alpha-1}u\|_2^2.
\end{equation*}

Finally, since $p'(\partial_x)\pop\in S^0_{1,0}$ and $|\phi_{\beta}(x)|\le c\japa^{\beta}$, we get
\begin{equation*}
\begin{split}
\Big|\int \phi_{2\alpha-2}(x)p'(\partial_x)\pop u u dx\Big| &\le \|\phi_{1\alpha-1}(x)p'(\partial_x)\pop u\|_2
\|\phi_{\alpha-1}(x)u\|_2\\
&\le c\|\japa^{\alpha-1}(x)u\|_2^2.
\end{split}
\end{equation*}

Collecting the information above in \eqref{decay-star} and using Gronwall's inequality we obtain the
desired result.

\begin{remark} To justify this formal argument we can regularize the data (and use the continuous dependence
on the solution upon the data) and truncate the weights $\japa^{2\alpha}$ as
\begin{equation*}
\varphi_{\alpha,N}(x)=
\begin{cases}
\japa^{2\alpha}, \hskip25pt |x|\le N,\\
\langle 2N\rangle^{2\alpha}, \hskip15pt |x|\ge 3N,
\end{cases}
\end{equation*}
with $\varphi_{\alpha,N}$ smooth, even, nondecreasing for $x\ge 0$, and check that 
all the estimates above are independent of $N$ and the regularization on the data.
This allows us to pass to the limit to get the result.

\end{remark}

\section{Proof of Theorem \ref{main-weights-small}}

For $\beta>0$ define $\phi_{\beta}\in C^{\infty}(\R)$ as $\phi_{\beta}(x)=x\japa^{\beta-1}$, where $\japa=(1+x^2)^{\frac12}$.
Notice that
\begin{enumerate}
\item $|\phi_{\beta}(x)|\le \japa^{\beta}$.
\item There exist constants $c_{\beta}, \tilde{c}_{\beta}>0$ such that
\begin{equation}\label{weights-1-s}
c_{\beta} \japa^{\beta-1}\le \phi'_{\beta}(x)\le \tilde{c}_{\beta} \japa^{\beta-1}.
\end{equation}
\end{enumerate}

From the equation \eqref{ILW} one formally gets the identity
\begin{equation}\label{weights-2-s}
\begin{split}
\frac12\frac{d}{dt}& \int u^2(x,t)\phi_{2\alpha}(x)dx+\int \lop\partial_xu\, u\phi_{2\alpha}dx+\int u\partial_xu\,u\phi_{2\alpha}\,dx\\
&=A_1(t)+A_2(t)+A_3(t)=0,
\end{split}
\end{equation}
with $\alpha\in[1/6,1/2]$. After integrating in the time interval $[t_1, t_2]$ it follows that
\begin{equation}\label{weights-3-s}
\int_{t_1}^{t_2} A_1(t)\,dt \lesssim \,\| |x|^{\alpha}u(\cdot,t_2)\|_2^2+\||x|^{\alpha}u(\cdot, t_1)\|_2^2 <\infty
\end{equation}
by hypothesis \eqref{decay-hyp1}.
Integration by parts  leads to
\begin{equation}\label{weights-4-s}
 A_3(t)=-\frac13\int u^3(x,t) \phi_{2\alpha}' (x)\,dx.
 \end{equation}
 Thus, from the assumption  on $\alpha$, one sees that
 \begin{equation}\label{weights-5-s}
 \begin{split}
 |A_3(t)|&\le c \int |u(x,t)|^3\japa^{2\alpha-1} \,dx\\
 &\le c\|u(\cdot,t)\|_3^3  \le c\|u(\cdot,t)\|_{\frac16,2}^3\in L^{\infty}([-T,T])
 \end{split}
 \end{equation}
 by hypothesis. Therefore, after integration in the time interval $[t_1, t_2]$ it follows that
 \begin{equation}\label{weights-6-s} 
 \int_{t_1}^{t_2} |A_3(t)|\,dt <\infty.
\end{equation}
Finally, we consider $A_2$ in \eqref{weights-2-s}. Using that $\lop\partial_x$ is skew symmetric
 \begin{equation}\label{weights-7-s}
 \begin{split}
A_2(t)&= \int \lop\partial_xu\,u\phi_{2\alpha}dx=-\int u\lop\partial_x(u\phi_{2\alpha})dx\\
&=-\int u\phi_{2\alpha}\lop\partial_xudx-\int u\,[\lop\partial_x; \phi_{2\alpha}]udx.
\end{split}
\end{equation}
Thus 
\begin{equation}\label{weights-8-s}
2A_2(t)= -\int u\,[\lop\partial_x; \phi_{2\alpha}]u.
\end{equation}
Using formula \eqref{lem3-comm} one can see that
\begin{equation}\label{weights-9-s}
\begin{split}
[\lop\partial_x; \phi_{2\alpha}]&= c_1 \phi_{2\alpha}'(x)\Omega'(\partial_x)+c_2 \phi_{2\alpha}''(x)\Omega''(\partial_x)+ \mathcal{R}_3\\
&=A_{21}+A_{22}+A_{23}.
\end{split}
\end{equation}
Since $\alpha\in [1/6,1/2]$ and $\Omega''(\partial_x)$  is an operator of order zero, from our commutator estimates in section 2, one sees that 
$\phi_{2\alpha}''(x)\Omega''(\partial_x)$ and $ \mathcal{R}_3$ are bounded in $L^2$. This takes care of the estimate for $A_{22}$ and $A_{23}$ in 
\eqref{weights-9-s}. So it remains to consider $ A_{21}$. Thus,
\begin{equation}\label{weights-11-s}
\begin{split}
A_{21}&(t)=-c_1\int \phi_{2\alpha}'(x)\Omega'(\partial_x)u\,u dx=-c_1\int \phi_{2\alpha}'(x)(\pop)^2u\,u dx\\
&=-c_1\int \pop(\phi_{2\alpha}'(x)\pop u)\,u dx+c_1\int [\pop; \phi_{2\alpha}'] \pop u\,u dx\\
&= A_{211}(t)+A_{212}(t).
\end{split}
\end{equation}
The commutator \eqref{lem3-comm} leads to 
\begin{equation*}
[\pop; \phi_{2\alpha}']\pop u\, u= \phi_{2\alpha}''p'(\partial_x)\pop u \, u +\mathcal{R}_2u u.
\end{equation*}
Since $p'(\partial_x)\pop$ is an operator of order zero, from the assumption on $\alpha$, we deduce that
\begin{equation*}
\big|\int  \phi_{2\alpha}''p'(\partial_x)\pop u u dx\big| \le \|u\|_2^2 \hskip10pt \text{and}\hskip5pt |\int \mathcal{R}_2u u dx|\le c\|u\|_2^2. 
\end{equation*}
Thus
\begin{equation*}
|A_{212}(t)|\lesssim \|\japa^{\frac32}u\|^2_2.
\end{equation*}

Finally, since $\pop$ is odd
\begin{equation*}
A_{211}(t)=c_1\int \phi_{2\alpha}'(x)\pop u)\,\pop u \simeq \int |\japa^{\alpha-\frac12}\pop u|^2\,dx.
\end{equation*}

Hence, collecting the above information and integrating in the time interval $[t_1, t_2]$ one has that
\begin{equation}\label{weights-20-s}
\int_{t_1}^{t_2}\int |\japa^{\alpha-\frac12} \pop u(x,t)|^2\,dxdt <\infty.
\end{equation}

Reapplying the previous argument with $\japa^{2\alpha}$ instead of $\phi_{2\alpha}$ one gets that
\begin{equation}\label{weights-21-s}
\japa^{\alpha}u(t)\in L^2(\R)\;\;\;\;\;\;\;t\in[0,T].
\end{equation}

Taking $\mu\in C^{\infty}_0(\R)$, with $\mu(x)=1,\;|x|\leq 1$ and $\mu(x)=0,\;|x|\geq 2$ we write
\begin{equation}
\label{weights-22-s}
\japa^{\alpha-1/2}J^{1/2}u=\japa^{\alpha-1/2}J^{1/2}(\mu(\partial_x)+(1-\mu(\partial_x))u.
\end{equation}

Since 
$$
J^{1/2}\mu(\partial_x),\;\;\;\;\;\;\;J^{1/2}(1-\mu(\partial_x))/p(\partial_x),
$$
 are operators of order zero, from Lemma \ref{pseudo+weights}, one has that $\japa^{\alpha-1/2}J^{1/2}u(t) \in L^2(\R)$ at each time $t$ where $\japa^{\alpha-\frac12} \pop u(t)\in L^2(\R)$. Therefore, from \eqref{weights-20-s} one gets that $\japa^{\alpha-1/2}J^{1/2}u(t) \in L^2(\R)$ a.e. in $[0,T]$. 
 
 This combined with \eqref{weights-21-s}, and the interpolation 
result in Lemma \ref{lem2b-interpolation}, tells us that for some $t_0\in[0,T]$ one has that $J^{\alpha}u(t_0)\in L^2(\R)$ which yields the desired result.

\end{proof}

\section{Proof of Theorem \ref{main-weights}}

\subsection*{Case $\frac32<s<\alpha<2$} \hskip15pt

\vspace{3mm}

As in the proof of the previous theorem we define $\phi_{\beta}\in C^{\infty}(\R)$ as $\phi_{\beta}(x)=x\japa^{\beta-1}$, where $\japa=(1+x^2)^{\frac12}$.
Notice that
\begin{enumerate}
\item $|\phi_{\beta}(x)|\le \japa^{\beta}$.
\item There exist constants $c_{\beta}, \tilde{c}_{\beta}>0$ such that
\begin{equation}\label{weights-1}
c_{\beta} \japa^{\beta-1}\le \phi'_{\beta}(x)\le \tilde{c}_{\beta} \japa^{\beta-1}.
\end{equation}
\end{enumerate}

From the equation \eqref{ILW} we formally get the identity
\begin{equation}\label{weights-2}
\begin{split}
\frac12\frac{d}{dt}& \int u^2(x,t)\phi_{2\alpha}(x)\,dx+\int \lop\partial_xu\, u\phi_{2\alpha}\,dx+\int u\partial_xu\,u\phi_{2\alpha}\,dx\\
&=A_1(t)+A_2(t)+A_3(t)=0.
\end{split}
\end{equation}

After integrating in the time interval $[t_1, t_2]$ one has
\begin{equation}\label{weights-3}
\int_{t_1}^{t_2} A_1(t)\,dt \lesssim \,\| |x|^{\alpha}u(\cdot,t_2)\|_2^2+\||x|^{\alpha}u(\cdot, t_1)\|_2^2 <\infty
\end{equation}
by hypothesis.

Integration by parts yields
\begin{equation}\label{weights-4}
 A_3(t)=-\frac13\int u^3(x,t) \phi_{2\alpha}' (x)\,dx.
 \end{equation}
 Thus
 \begin{equation}\label{weights-5}
 \begin{split}
 |A_3(t)|&\le c \int |u(x,t)|^3\japa^{2\alpha-1} \,dx\\
 &\le c\|u\|_{\infty}\|\japa^{\frac32}u(\cdot,t)\|_2^2 \in L^{\infty}([-T,T]).
 \end{split}
 \end{equation}
 Therefore, after integration in the time interval $[t_1, t_2]$ it follows that
 \begin{equation}\label{weights-6} 
 \int_{t_1}^{t_2} |A_3(t)|\,dt <\infty.
\end{equation}

Finally, we consider $A_2$ in \eqref{weights-2}. Using that $\lop\partial_x$ is skew symmetric
 \begin{equation}\label{weights-7}
 \begin{split}
A_2(t)&= \int \lop\partial_xu\,u\phi_{2\alpha}dx=-\int u\lop\partial_x(u\phi_{2\alpha})dx\\
&=-\int u\phi_{2\alpha}\lop\partial_xudx-\int u\,[\lop\partial_x; \phi_{2\alpha}]udx.
\end{split}
\end{equation}
Thus 
\begin{equation}\label{weights-8}
2A_2(t)= -\int u\,[\lop\partial_x; \phi_{2\alpha}]u dx.
\end{equation}
Using formula \eqref{lem3-comm} one can see that
\begin{equation}\label{weights-9}
\begin{split}
[\lop\partial_x; \phi_{2\alpha}]&= c_1 \phi_{2\alpha}'(x)\Omega'(\partial_x)+c_2 \phi_{2\alpha}''(x)\Omega''(\partial_x)\\
&\hskip10pt+c_3 \phi_{2\alpha}^{(3)}(x)\Omega^{(3)}(\partial_x)+ c_4 \phi_{2\alpha}^{(4)}(x)\Omega^{(4)}(\partial_x)+ \mathcal{R}_5\\
&=A_{21}+A_{22}+A_{23}+A_{24}+A_{25}.
\end{split}
\end{equation}

Since the operators $\Omega'(\partial_x), \dots, \Omega^{(4)}(\partial_x)$ are of order zero by  the hypothesis on $\alpha$ ($\alpha<2$), 
Remark \ref{rem-main} (7), and the hypothesis
\begin{equation}\label{weights-10}
\begin{split}
\Big|\int \phi_{2\alpha}''(x)\Omega''(\partial_x)u\,u dx\Big|&\le \int |\japa^{\alpha-1}\Omega''(\partial_x)u||\japa^{\alpha-1}u dx|\\
&\le \|\japa^{\alpha-1}u\|_2^2\le \|\japa^{\frac32}u\|_2^2  \in L^{\infty}([-T,T]).
\end{split}
\end{equation}
Since $\phi_{2\alpha}^{(5)}\in L^2(\R)$, from \eqref{lem3-comm-1} $\|A_{25}u\|_2\le c\|u\|_2$ and the terms 
$A_{23}, A_{24}$ can be handled as in \eqref{weights-10} then it remains  to consider $A_{21}$.

\begin{equation}\label{weights-11}
\begin{split}
A_{21}&(t)=-c_1\int \phi_{2\alpha}'(x)\Omega'(\partial_x)u\,u dx=-c_1\int \phi_{2\alpha}'(x)\pop^2u\,u dx\\
&=-c_1\int \pop(\phi_{2\alpha}'(x)\pop u)\,u dx+c_1\int [\pop; \phi_{2\alpha}'] \pop u\, u dx\\
&= A_{211}(t)+A_{212}(t).
\end{split}
\end{equation}

The commutator \eqref{lem3-comm} leads to 
\begin{equation*}
[\pop; \phi_{2\alpha}']\pop u\, u= \phi_{2\alpha}''p'(\partial_x)\pop u \, u +\mathcal{R}_2u u.
\end{equation*}
Since $p'(\partial_x)\pop$ is an operator of order zero we deduce that
\begin{equation*}
\big|\int  \phi_{2\alpha}''p'(\partial_x)\pop u u dx\big| \le \|\japa^{\alpha-1}u\|_2^2 \hskip10pt \text{and}\hskip5pt |\int \mathcal{R}_2u u dx|\le c\|u\|_2^2. 
\end{equation*}
Thus
\begin{equation*}
|A_{212}(t)|\lesssim \|\japa^{\frac32}u\|^2_2.
\end{equation*}

Finally, since $\pop$ is odd
\begin{equation*}
A_{211}(t)=c_1\int \phi_{2\alpha}'(x)\pop u)\,\pop u dx \simeq \int |\japa^{\alpha-\frac12}\pop u|^2\,dx.
\end{equation*}

Hence, collecting the above information and integrating in the time interval $[t_1, t_2]$ one has that
\begin{equation}\label{weights-20}
\int_{t_1}^{t_2}\int |\japa^{\alpha-\frac12} \pop u(x,t)|^2\,dxdt <\infty.
\end{equation}

\medskip

Reapplying the argument above with $\japa^{2\alpha}$ instead of $\phi_{2\alpha}$ and using \eqref{weights-20}
it follows that

\begin{equation}\label{weights-21}
\japa^{\alpha-\frac12}\pop u \in L^{\infty}([t_1, t_2] : L^2(\R)).
\end{equation}

In particular, there exist $\hat{t}_1, \hat{t}_2$,  $\hat{t}_1\neq \hat{t}_2$,  $\hat{t}_1, \hat{t}_2\in (t_1, t_2)$ such that
\begin{equation}\label{weights-22}
\japa^{\alpha-\frac12}\pop u(x, \hat{t}_j)  \in L^2(\R) \hskip5pt j=1,2.
\end{equation}

Working in the time interval $(\hat{t}_1, \hat{t}_2)$ and using the equation \eqref{ILW}  one gets the identity
\begin{equation}\label{weights-23}
\begin{split}
\frac12\frac{d}{dt}& \int (\pop u)^2\phi_{2\alpha-1}(x)\,dx+\int \lop\partial_x\pop u\, \pop u\,\phi_{2\alpha-1}\,dx\\
&+\int \pop(u\partial_xu)\,\pop u\,\phi_{2\alpha-1}\,dx =B_1(t)+B_2(t)+B_3(t)=0.
\end{split}
\end{equation}

As before by hypothesis \eqref{weights-22} one has that
\begin{equation}\label{weights-24}
\int_{\hat{t}_1}^{\hat{t}_2}\!\! \! B_1(t)\,dt \lesssim \|\japa^{\alpha-\frac12}\pop u(\cdot,\hat{t}_1)\|_2^2+  \|\japa^{\alpha-\frac12}\pop u(\cdot,\hat{t}_2)\|_2^2 <\infty.
\end{equation}

Next, set $\pop u=v$, since $\lop\partial_x$ is skew symmetric 
\begin{equation*}
\begin{split}
B_2(t)&=\int \lop\partial_xv\, v\, \phi_{2\alpha-1}\,dx= -\int v \lop\partial_x (v\phi_{2\alpha-1})\,dx\\
&=-\int v\lop \partial_x v \,\phi_{2\alpha-1}dx-\int v\,[\lop\partial_x; \phi_{2\alpha-1}]vdx.
\end{split}
\end{equation*}

Then
\begin{equation*}
2B_2(t)= -\int v\,[\lop\partial_x; \phi_{2\alpha-1}]v dx
\end{equation*}
but from formula \eqref{lem3-comm} 
\begin{equation}\label{weights-22b}
\begin{split}
[\lop\partial_x; \phi_{2\alpha-1}]&= c_1 \phi_{2\alpha-1}'(x)\Omega'(\partial_x)+c_2 \phi_{2\alpha-1}''(x)\Omega''(\partial_x)\\
&\hskip10pt+c_3 \phi_{2\alpha-1}^{(3)}(x)\Omega^{(3)}(\partial_x)+ \mathcal{R}_4 \\
&=B_{21}+B_{22}+B_{23}+B_{24}.
\end{split}
\end{equation}

One has that
\begin{equation*}
\begin{split}
\|\phi_{2\alpha-1}''\Omega''(\partial_x)v\|_2 &\simeq \|\japa^{2\alpha-3} \Omega''(\partial_x)v\|_2\\
&\le c\|\japa^{2\alpha-3}\,v\|_2\\
&\le c\|\japa^{\alpha-\frac12}v\|_2 \simeq \| \japa^{\alpha-\frac12}\pop u\|_2 <\infty.
\end{split}
\end{equation*}
(see \eqref{weights-21}) since $2\alpha-3\le \alpha-\frac12$  ($\iff \alpha<2 \le 3-\frac12$).

Since $\phi_{2\alpha-1}^{(4)}\sim \japa^{2\alpha-5}\in L^2(\R)$,
\begin{equation*}
\|\mathcal{R}_4 v\|_2\le \|v\|_2 \le \|\pop u\|_2 \le \|u\|_{1,2}.
\end{equation*}
The same argument applies to $B_{23}$. So it remains to analise $B_{21}$.

\begin{equation}\label{weights-23b}
\begin{split}
&B_{21}(t)= -c \int \phi_{2\alpha-1}'\pop \pop v \, v dx\\
&=-c_1\int \pop (\phi_{2\alpha-1}'\pop v)v dx - c_1 \int[\pop; \phi_{2\alpha-1}']\pop v \, v dx\\
&= B_{211}+B_{212}.
\end{split}
\end{equation}

Now, by \eqref{lem3-comm} we get
\begin{equation*}
[\pop; \phi_{2\alpha-1}']\pop = \phi_{2\alpha-1}'' p'(\partial_x) \pop + \phi_{2\alpha-1}^{(3)}p''(\partial_x)\pop+\dots
+\mathcal{R}_m\,\pop.
\end{equation*}

We will only consider the first term on the right side of the identity above since the other ones are simpler or similar.
Using that  $p'(\partial_x) \pop$ is an operator of order zero we get
\begin{equation*}
\begin{split}
|\int \phi_{2\alpha-1}'' p'(\partial_x) \pop v \, v  dx|&\simeq |\int \japa^{2\alpha-3}p'(\partial_x) \pop v \, v dx|\\
&\simeq |\int \japa^{\alpha-\frac32}p'(\partial_x) \pop v \japa^{\alpha-\frac32}v dx|\\
&\le \|\japa^{\alpha-\frac32}v\|_2^2\le  \|\japa^{\alpha-\frac32}\pop u\|_2^2\\
& \le \|\japa^{\alpha-\frac12}\pop u\|_2^2<\infty.
\end{split}
\end{equation*}

So it remains to bound $B_{21}$ in \eqref{weights-23b}. Because $\pop$ is odd we deduce that
\begin{equation*}
B_{211}=c_1\int \phi_{2\alpha-1}'(x) \pop v \pop v\,dx\equiv c_1\int |\japa^{\alpha-1}p^2(\partial_x) u|^2\,dx.
\end{equation*}

Next we shall estimate the nonlinear contribution.
\begin{equation}\label{weights-24b}
\begin{split}
&B_3(t)=\int \pop(u\partial_x u)\pop u \phi_{2\alpha-1} dx\\
&= -\!\int[\pop; \phi_{2\alpha-1}(x)]\partial_x\frac{(u^2)}{2} \,\pop u dx+\!\!\int\phi_{2\alpha-1}(x)u\partial_xu \,p^2(\partial_x)u dx\\
&=B_{31}+B_{32}.
\end{split}
\end{equation}

Consider first $B_{31}$. By using \eqref{lem3-comm} we get
\begin{equation*}
\begin{split}
[\pop; \phi_{2\alpha-1}(x)]&= c_1\phi_{2\alpha-1}'(x)p'(\partial_x)+c_2\phi_{2\alpha-1}''(x)p''(\partial_x)\\
&\hskip10pt+c_3\phi_{2\alpha-1}^{(3)}(x)p^{(3)}(\partial_x)+\mathcal{R}_4.
\end{split}
\end{equation*}
where $\|\mathcal{R}_4u\|_2\le \|u\|_2$ since $\phi_{2\alpha-1}^{(4)}(x)\in L^2(\R)$.

Now notice that
\begin{equation*}
\japa^{\alpha-\frac32}p'(\partial_x)\partial_x f= p'(\partial_x)\partial_x (\japa^{\alpha-\frac32}f)-[\japa^{\alpha-\frac32}, p'(\partial_x)\partial_x]f.
\end{equation*}
Then
\begin{equation*}
[ \japa^{\alpha-\frac32}; p'(\partial_x)\partial_x]f= -\{(\japa^{\alpha-\frac32})' (p'(\partial_x)\partial_x)'+ \cdots+\mathcal{R}_k\}.
\end{equation*}

By a familiar argument
\begin{equation*}
\|[\japa^{\alpha-\frac32}, p'(\partial_x)\partial_x] u^2\|_2^2\le c\|u^2\|^2_2\le c \|u\|_{\infty}^2\|u\|_2^2.
\end{equation*}

Finally, to conclude $B_{31}$
\begin{equation*}
\begin{split}
\|p'(\partial_x)\partial_x(\japa^{\alpha-\frac32}u^2)\|&\le c\|\japa^{\alpha-\frac32}u^2\|_{1,2}\\
&\le c\|\japa^{\alpha-\frac32}u^2\|_{2}+ c \|\partial_x(\japa^{\alpha-\frac32}u^2)\|_{2}\\
&\lesssim \|u\|_{\infty}\|\japa^{\alpha}u\|_{2}+\|\japa^{\alpha-\frac32}u\partial_x u\|_{2}\\
&\lesssim (\|u\|_{\infty}+\|\partial_xu\|_{\infty})\|\japa^{\alpha}u\|_{2}.
\end{split}
\end{equation*}

Now, we turn to $B_{32}$ in \eqref{weights-24b}.
\begin{equation*}
\begin{split}
|B_{32}(t)|&=\big|\int \phi_{2\alpha-1}(x) u\partial_xu\, p^2(\partial_x)u dx\big|\\\
&=\big|\int \phi_{\alpha-1}(x)p^2(\partial_x)u \,\phi_{\alpha}(x)u\partial_xu dx\big|\\
&\le \|\phi_{\alpha-1}(x)p^2(\partial_x)u\|_2(\|\partial_x u\|_{\infty}\|\phi_{\alpha}(x)u\|_2)\\
&\le \epsilon \|\japa^{\alpha-1}p^2(\partial_x)u\|_2^2+c_{\epsilon}\|\partial_xu\|_{\infty}^2\|\phi_{\alpha}(x)u\|_2^2.
\end{split}
\end{equation*}

The second term on the right hand side above is uniformly bounded in $[\hat{t}_1, \hat{t}_2]$. After integration in the time interval $[\hat{t}_1, \hat{t}_2]$,
we find that
\begin{equation*}
\int_{\hat{t}_1}^{\hat{t}_2} \int |\japa^{\alpha-1}p^2(\partial_x)u|^2\,dxdt\le M + \epsilon\int_{\hat{t}_1}^{\hat{t}_2}\int |\japa^{\alpha-1}p^2(\partial_x)u|^2\,dxdt
\end{equation*}
where $M$ is bounded. Then it implies that
\begin{equation}\label{weights-31}
\int_{\hat{t}_1}^{\hat{t}_2}\int |\japa^{\alpha-1}p^2(\partial_x)u|^2\,dxdt <\infty.   \hskip15pt \text{(smoothing effect!)}
\end{equation}

Reapplying the above argument with $\japa^{2\alpha-1}$ instead of $\phi_{2\alpha-1}$ and using \eqref{weights-1} it follows that
\begin{equation}\label{weights-32}
\japa^{\alpha-\frac12}p(\partial_x)u(x,t)\in L^{\infty}([\hat{t}_1, \hat{t}_2]:L^2(\R)).
\end{equation}

We observe that $\dfrac{\partial_x}{p^2(\partial_x)}\in S^0_{1,0}$. This implies that
\begin{equation*}
\|\japa^{\alpha-1}\partial_xu\|_2=\|\japa^{\alpha-1}\frac{\partial_x}{p^2(\partial_x)}p^2(\partial_x)u\|_2 
\le c\|\japa^{\alpha-1}p^2(\partial_x)u\|_2.
\end{equation*}
Thus, from this we have that the estimate \eqref{weights-31} holds with the operator $\partial_x$ instead of $p^2(\partial_x)$.

In particular, there exist $\hat{\hat{t}}_1, \hat{\hat{t}}_2 \in [\hat{t}_1,\hat{t}_2]$ such that
\begin{equation}\label{weights-33}
\begin{cases}
\begin{aligned}
&\japa^{\alpha-1} \partial_x u(x, \hat{\hat{t}}_j)\in L^2(\R) \\
\hskip3pt\text{and}\hskip3pt&\\
&\japa^{\alpha-1}\partial_x u\in L^2 ([\hat{t}_1,\hat{t}_2]: L^2(\R)).
\end{aligned}
\end{cases}
\end{equation}

\medskip

As in the previous analysis we use the equation \eqref{ILW} to get that
\begin{equation}\label{weights-34}
\begin{split}
\frac12\frac{d}{dt}& \int (\partial_x u)^2\phi_{2\alpha-2}(x) dx+\int \lop\partial_x\partial_xu \,\partial_x u\,\phi_{2\alpha-2}(x) dx\\
&+\int \phi_{2\alpha-2}(x)\partial_x(u\partial_xu)\,\partial_x u\,\,dx =D_1(t)+D_2(t)+D_3(t)=0.
\end{split}
\end{equation}

Then integration in the time interval $[\hat{\hat{t}}_1, \hat{\hat{t}}_2]$ leads to
\begin{equation*}
\int_{\hat{\hat{t}}_1}^{\hat{\hat{t}}_2} D_1(t)\,dt  \le \underset{j=1}{\overset{2}{\sum}} \|\japa^{\alpha-1}\partial_x u(\cdot, \hat{\hat{t}}_j) \|_2^2<\infty.
\end{equation*}

Also integration by parts gives
\begin{equation*}
D_3(t)\simeq \int \phi_{2\alpha-2}(x)\partial_xu\partial_xu\partial_xu dx+ \int  \phi_{2\alpha-2}'(x)\partial_xu\partial_xu u dx.
\end{equation*}
Hence
\begin{equation*}
|D_3(t)|\lesssim\|\japa^{\alpha-1}\partial_xu\|_2^2\|\partial_xu\|_{\infty}+ \|\japa^{\alpha}u\|_2\|\japa^{\alpha-1}\partial_xu\|_2\|\partial_xu\|_{\infty}.
\end{equation*}
which by \eqref{weights-33} and the hypothesis ($s>\frac32$), the terms on the right hand side are bounded after integration in the time interval
$[\hat{\hat{t}}_1, \hat{\hat{t}}_2] \subseteq [\hat{t}_1,\hat{t}_2]$.

Finally, $D_2$ with $w=\partial_xu$ can be written as:
\begin{equation*}
\begin{split}
D_2(t) &= \int \lop\partial_x w \, w\phi_{2\alpha-2}(x) dx=-\int w \lop\partial_x (w\,\phi_{2\alpha-2}(x)) dx\\
&\hskip10pt-\int w \lop\partial_x w\,\phi_{2\alpha-2}(x) dx-\int w [\lop;\phi_{2\alpha-2}(x)]w dx,
\end{split}
\end{equation*}
which implies that
\begin{equation*}
D_2(t)= -\int w [\lop;\phi_{2\alpha-2}(x)]wdx,
\end{equation*}
but \eqref{lem3-comm} yields
\begin{equation*}
[\lop;\phi_{2\alpha-2}(x)]= c_1\phi_{2\alpha-2}'(x)\Omega'(\partial_x)+ c_2\phi_{2\alpha-2}''(x)\Omega''(\partial_x)+\mathcal{R}_3.
\end{equation*}

Since $\phi_{2\alpha-2}'''\in L^2(\R)$, we deduce that $\|\mathcal{R}_3w\|_2\le c\|w\|_2\le c\|\partial_x u\|_2$. We also have that
$\phi_{2\alpha-2}''\in L^{\infty}(\R)$ and $\Omega''(\partial_x)$ is an order zero operator. Thus 
\begin{equation*}
\|\phi_{2\alpha-2}''(x)\Omega''(\partial_x)w\|_2\le \|w\|_2\le \|\partial_x u\|_2.
\end{equation*}

Finally, since $\phi_{2\alpha-2}'(x)\Omega'(\partial_x)=\japa^{2\alpha-3}\pop\pop$ it follows that
\begin{equation*}
\begin{split}
\int \japa^{2\alpha-3}\pop\pop w\cdot w dx&= \int \pop(\japa^{2\alpha-3}\pop w)\, w \,dx\\
&\hskip10pt+\int w[\pop, \japa^{2\alpha-3}]\pop w \,dx\\
&= D_{21}(t)+D_{22}(t).
\end{split}
\end{equation*}

We start by estimating $D_{22}$. By formula \eqref{lem3-comm} we get
\begin{equation*}
[\pop, \japa^{2\alpha-3}]= c_1(\japa^{2\alpha-3})'p'(\partial_x)+\mathcal{R}_2
\end{equation*}
whose terms on the right hand side are both bounded in $L^2(\R)$.

Hence, it remains to consider $D_{21}$. Recall that $\pop$ is odd, then
\begin{equation*}
\begin{split}
D_{21}(t)=-\int\japa^{\alpha-\frac32} \pop w\, \pop w \,dx\lesssim \int \big| \japa^{2\alpha-3} \pop w|^2\,dx.
\end{split}
\end{equation*}

Thus after integration in the time interval ones concludes that
\begin{equation}\label{weights-40}
\int_{\hat{\hat{t}}_1}^{\hat{\hat{t}}_2} \int |\japa^{\alpha-\frac32}\pop\partial_x u|^2\,dxdt<\infty
\end{equation}
and reapplying the argument with $\japa^{2\alpha-2}$ instead of $\phi_{2\alpha-2}(x)$ we deduce that
\begin{equation}\label{weights-41}
\japa^{\alpha-1}\partial_xu(x,t) \in L^{\infty}([\hat{\hat{t}}_1, \hat{\hat{t}}_2]; L^2(\R)).
\end{equation}

\medskip

\subsection*{CLAIM} From our previous results it follows that

\vspace{2mm}

\begin{equation}\label{weights-42}
\begin{cases}
\begin{aligned}
&\hskip1cm\japa^{\alpha-\frac32} J^{\frac32} u \in L^2((\hat{\hat{t}}_1, \hat{\hat{t}}_2)\times \R)\\
&\hskip5pt\text{and there exist} \hskip5pt \bar{t}_1,\bar{t}_2\in (\hat{\hat{t}}_1, \hat{\hat{t}}_2)\hskip5pt\text{such that} \\
&\hskip1cm\japa^{\alpha-\frac32}J^{\frac32}u(x, \bar{t}_j)\in L^2(\R),\; j=1,2.\\
&\hskip5pt\text{Also}\\
&\hskip1cm\japa^{\alpha-1}Ju\in L^{\infty} ([\hat{\hat{t}}_1,\hat{\hat{t}}_2]: L^2(\R)).
\end{aligned}
\end{cases}
\end{equation}

\begin{proof} We shall write
\begin{equation*}
\japa^{\alpha-\frac32}J^{\frac32}u=-[J^{\frac32};\japa^{\alpha-\frac32}]u- J^{\frac32}(\japa^{\alpha-\frac32}u).
\end{equation*}
By  formula \eqref{lem3-comm} it follows that
\begin{equation*}
[J^{\frac32};\japa^{\alpha-\frac32}]u\simeq \japa^{\alpha-\frac52} J^{\frac12}u +\mathcal{R}u
\end{equation*}
where $\|\mathcal{R}u\|_2\le c\|u\|_2<\infty$ and since $(\alpha<2)$
\begin{equation*}
\|\japa^{\alpha-\frac52} J^{\frac12}u\|_2\le c\|J^{\frac12}u\|_2\le c\|u\|_{\frac12,2}<\infty.
\end{equation*}

Now,
\begin{equation*}
\|J^{\frac32}(\japa^{\alpha-\frac32}u)\|_2^2\le \|\japa^{\alpha-\frac32}u\|_2^2+\|\partial_x\pop(\japa^{\alpha-\frac32}u)\|_2^2
\end{equation*}
because $(1+\xi^2)^{\frac32}\le 1+|p(\xi)\xi|^2$.

On the other hand,
\begin{equation*}
\|\japa^{\alpha-\frac32}u\|_2\le \|\japa^{\alpha}u\|_2
\end{equation*}
and
\begin{equation}\label{weights-42a}
\begin{split}
\|\partial_x\pop(\japa^{\alpha-\frac32}u)\|_2\lesssim \|\pop(\japa^{\alpha-\frac32}\partial_xu)\|_2+\|\pop(\japa^{\alpha-\frac52}u)\|_2.
\end{split}
\end{equation}

Now applying the commutator \eqref{comm-go-1}, we deduce
\begin{equation*}
\begin{split}
\|\pop&(\japa^{\alpha-\frac52}u)\|_2 \le \|J^{\frac12}(\japa^{\alpha-\frac52}u)\|_2\\
&\lesssim \|\japa^{\alpha-\frac52}\|_{\infty}\|J^{\frac12}u\|_2+\|J^{\frac12}(\japa^{\alpha-\frac52})\|_2\|u\|_{\infty}<\infty.
\end{split}
\end{equation*}

The first term on the right hand side of \eqref{weights-42a} can be written as
\begin{equation*}
\begin{split}
\pop (\japa^{\alpha-\frac32}\partial_x u)=[\pop ; \japa^{\alpha-\frac32}] \partial_xu + \japa^{\alpha-\frac32}\pop\partial_x u.
\end{split}
\end{equation*}
Since the last term on the right hand side is in $L^2(\R)$ by hypothesis, it remains to estimate the first one. 

By formula \eqref{lem3-comm}
\begin{equation*}
\begin{split}
[\pop ; \japa^{\alpha-\frac32}] \partial_xu\simeq \japa^{\alpha-\frac52}p'(\partial_x)\partial_xu +
\japa^{\alpha-\frac72}p''(\partial_x)\partial_xu +\mathcal{R}_3\partial_xu
\end{split}
\end{equation*}
with
\begin{equation*}
\begin{split}
&\|\mathcal{R}_3\partial_xu\|_2\le \|\partial_xu\|_2,\\
&\|\japa^{\alpha-\frac52}p'(\partial_x)\partial_xu\|_2\le c\|p'(\partial_x)\partial_xu\|_2\le c\|u\|_{\frac12,2},\\
&\|\japa^{\alpha-\frac72}p''(\partial_x)\partial_xu\|_2\le c\|p''(\partial_x)\partial_xu\|_2\le c\|u\|_2.
\end{split}
\end{equation*}

\medskip

This completes the proof of our claim.

\end{proof}

Using \eqref{weights-42}, we write for $\theta=\alpha-\frac32\in (0,\frac12)$. Then 
\begin{equation}\label{weights-33b}
\begin{split}
\frac12\frac{d}{dt}& \int (J^{\frac32} u)^2\phi_{2\theta}(x) dx+\int \lop\partial_xJ^{\frac32}u \, J^{\frac32}u\,\phi_{2\theta}(x) dx\\
&+\int \phi_{2\theta}(x)J^{\frac32}(u\partial_xu)\,J^{\frac32}u\,\,dx =E_1(t)+E_2(t)+E_3(t)=0.
\end{split}
\end{equation}

From \eqref{weights-42} we obtain that
\begin{equation*}
\begin{split}
\int_{\bar{t}_1}^{\bar{t}_2} E_1(t)\,dt  \le \underset{j=1}{\overset{2}{\sum}} \|\japa^{\alpha-\frac32}J^{\frac32} u(\cdot,\bar{t}_j) \|_2^2. 
\end{split}
\end{equation*}

Next, we estimate $E_2(t)$. Setting $z=J^{\frac32}u$ and using that $\lop\partial_x$ is skew symmetric, we have
\begin{equation*}
\begin{split}
E_2(t)&=\int \lop\partial_x z \, z\,\phi_{2\theta}(x) dx=-\int z \lop\partial_x (z\,\phi_{2\theta}(x)) dx\\
&=-\int z \lop\partial_x z\,\phi_{2\theta}(x) dx-\int z [\lop\partial_x;\phi_{2\theta}(x)]z dx.
\end{split}
\end{equation*}
Thus 
\begin{equation*}
2E_2(t)= -\int z [\lop\partial_x;\phi_{2\theta}(x)]z dx.
\end{equation*}

A familiar argument yields
\begin{equation*}\label{weights-33c}
[\lop\partial_x;\phi_{2\theta}(x)]=c_1 \japa^{2\theta-1}\Omega'(\partial_x)+c_2(\japa^{2\theta-1})'\Omega''(\partial_x)+\mathcal{R}_3.
\end{equation*}
with
\begin{equation*}
\|\mathcal{R}_3z\|_2\le c \|z\|_2\le c \|J^{\frac32}u\|_2
\end{equation*}
and
\begin{equation*}
\|(\japa^{2\theta-1})'\Omega''(\partial_x)z\|_2\le c \|\Omega''(\partial_x)z\|_2\le c\|z\|_2\le c\|J^{\frac32}u\|_2.
\end{equation*}

Finally, we estimate
\begin{equation*}
\begin{split}
\int &\japa^{2\theta-1}\Omega'(\partial_x)z\,z dx=\int  \japa^{2\theta-1}\pop \pop z\,z dx\\
&=-\int \pop z\,\pop(\japa^{2\theta-1}z) dx\\
&=-\int \pop z\cdot \pop z \japa^{2\theta-1} dx-\int \pop z \cdot [\pop; \japa^{2\theta-1}]z dx\\
&=-\int |\japa^{\theta-1/2}\pop z|^2 dx+\int z\pop [\pop; \japa^{2\theta-1}]z dx\\
&=-\int |\japa^{\theta-1/2}\pop J^{\frac32}u|^2 dx+F_1
\end{split}
\end{equation*}

By applying the estimate (2.36) in \cite{MPS} it follows that
\begin{equation*}
\begin{split}
F_1 &\le c\|z\|_2\|\pop [\pop; \japa^{2\theta-1}]z\|_2\\
&\le c \|z\|_2\|(\japa^{2\theta-1})'\|_{1,2}\|z\|_2\le c\|z\|_2\simeq c\|J^{\frac32}u\|_2
\end{split}
\end{equation*}
since  $(\japa^{2\theta-1})'\simeq \japa^{2\theta-2}\le \japa^{-1}\in L^2(\R)$.

\medskip

Finally, we consider $E_3(t)$ in \eqref{weights-33b}. 

We first write $\phi_{2\theta}(x)=x\japa^{2\theta-1}=x\japa^{\theta-1}\japa^{\theta}=\lambda_1(x)\lambda_2(x)$.
\begin{equation*}
\begin{split}
E_3(t)&=\int \phi_{2\theta}(x)J^{\frac32}(u\partial_xu)\,J^{\frac32}u\,\,dx\\
&=\int \lambda_1(x)J^{\frac32}(u\partial_xu)\,J^{\frac32}u \lambda_2(x)\,dx\\
&=\int J^{\frac32}(\lambda_1(x)u\partial_xu)\,J^{\frac32}u \lambda_2(x) dx\\
&\hskip10pt-\int [J^{\frac32};\lambda_1(x)]\partial_x(u^2)\lambda_2(x)J^{\frac32}u dx \\
&=G_1(t)+G_2(t).
\end{split}
\end{equation*}

Now
\begin{equation*}
\begin{split}
G_2(t)&\le \|[J^{\frac32}; \lambda_1(x)]\partial_x (u^2)\|^2_2+\|\lambda_2(x)J^{\frac32}u\|^2_2\\
&\lesssim \|[J^{\frac32}; \lambda_1(x)]\partial_x (u^2)\|_2^2+\|\japa^{\theta}J^{\frac32}u\|_2^2
\end{split}
\end{equation*}
since $|\lambda_2(x)|\sim \japa^{\theta}$.  After integrating in the time interval the last term is bounded.

Next, we need to estimate $\|[J^{\frac32}; \lambda_1(x)]\partial_x (u^2) \|_2^2$. To do so, we use first the commutator formula \eqref{lem3-comm} to yield
\begin{equation*}
[J^{\frac32}; \lambda_1(x)]\partial_x=\lambda_1'(x) \tilde Q_{0,1}(\partial_x) J^{\frac12}\partial_x+ \lambda_1''(x) \tilde Q_{0,2}(\partial_x)J^{-\frac12}\partial_x
+\mathcal{R}\partial_x,
\end{equation*}
with $\tilde Q_{0,1}, \tilde Q_{0,2}\in S^0_{1,0}$. Since $\lambda_1^{(3)}\in L^2(\R)$, we have $\|\mathcal{R}\partial_x(u^2)\|_2 \le c\|u\|_{\infty}\|\partial_x u\|_2$.
Then
\begin{equation*}
\|\lambda_1''(x) \tilde Q_{0,2}(\partial_x)J^{-\frac12}\partial_x(u^2)\|_2\le c \|J^{-\frac12}\partial_x(u^2)\|_2\le c\|u\|_{\infty}\|J^{\frac12}u\|_2
\end{equation*}
because $\lambda_1''\in L^{\infty}(\R)$, and
\begin{equation*}
\begin{split}
\|\lambda_1'(x) \tilde Q_{0,1}(\partial_x)J^{\frac12}\partial_x(u^2)\|_2&\simeq \|\japa^{\theta-1}J^{\frac12}\partial_x(u^2)\|_2\le
\|\japa^{\theta-1}\|_{\infty}\|J^{\frac12}\partial_x(u^2)\|_2\\
&\le c\|J^{\frac32}(u^2)\|_2\le
c\|u\|_{\infty}\|J^{\frac32}u\|_2
\end{split}
\end{equation*}
where above we have used the commutator estimate  \eqref{kp-comm}.

It remains to bound $G_1(t)$. We write
\begin{equation*}
\begin{split}
G_1(t)&=\int J^{\frac32}(\lambda_1(x)u\partial_xu)\lambda_2(x)J^{\frac32}u dx\\
&= \int \lambda_1(x)uJ^{\frac32}\partial_xu\lambda_2(x)J^{\frac32}u dx\\
&\hskip10pt+\int[J^{\frac32};\lambda_1(x)u]\partial_xu\,\lambda_2(x)J^{\frac32}u dx\\
&=G_{11}(t)+G_{12}(t).
\end{split}
\end{equation*}

By integration by parts
\begin{equation*}
\begin{split}
G_{11}(t)&=-\frac12 \int \big(J^{\frac32}u\big)^2\partial_x(\lambda_1(x)\lambda_2(x)u) dx\\
&=\frac12 \int \big(J^{\frac32}u\big)^2\partial_xu\lambda_1(x)\lambda_2(x) dx\\
&\hskip10pt-\frac12\int \big(J^{\frac32}u\big)^2u(\lambda_1(x)\lambda_2(x))'\, dx
\end{split}
\end{equation*}
Since $(\lambda_1(x)\lambda_2(x))'\in L^{\infty}(\R)$, it follows that
\begin{equation*}
\begin{split}
|G_{11}(t)|&\le \|\partial_xu\|_{\infty}\int |\japa^{\theta}J^{\frac32}u|^2\,dx+ \|u\|_{\infty}\int |J^{\frac32}u|^2\,dx.
\end{split}
\end{equation*}

Finally, we consider $G_{12}$. We first apply the generalized commutator \eqref{kp-comm}
to lead to
\begin{equation}\label{weights-44}
\begin{split}
\|[J^{\frac32};\lambda_1(x)u]\partial_xu\|_2&\le c\big(\|\partial_x(\lambda_1(x)u)\|_{\infty^{-}}\|J^{\frac32-1}\partial_xu\|_{2^{+}}\\
&\hskip15pt+\|\partial_xu\|_{\infty}\|J^{\frac32}(\lambda_1(x)u)\|_2\big)
\end{split}
\end{equation}
where $\;\dfrac12=\dfrac{1}{\infty^{-}}+\dfrac{1}{2^{+}}$.

\medskip

By Sobolev embedding
\begin{equation}\label{weights-45}
\begin{split}
\|\partial_x(\lambda_1(x)u)\|_{\infty^{-}} &\lesssim \|J^{\frac32}(\japa^{\theta}u)\|_2\\
\text{and}\hskip15pt&\\
\|J^{\frac32-1}\partial_xu\|_{2^{+}}\|&\lesssim \|J^{\frac32^{+}}u\|_2\le \|J^su\|_2.
\end{split}
\end{equation}

It remains to bound  $\|J^{\frac32}(\japa^{\theta}u)\|_2$.

We write
\begin{equation*}
J^{\frac32}(\japa^{\theta}u)=\japa^{\theta}J^{\frac32}u+[J^{\frac32}; \japa^{\theta}]u
\end{equation*}
We observe that $\|\japa^{\theta}J^{\frac32}u\|_2^2$ will be bounded after integration in the time interval.

Employing formula \eqref{lem3-comm} we get
\begin{equation*}
[J^{\frac32}; \japa^{\theta}]u=( \japa^{\theta})'Q_{0,1}(\partial_x)J^{\frac12}u+( \japa^{\theta})''Q_{0,2}(\partial_x)J^{-\frac12}u +\mathcal{R}_3u.
\end{equation*}
with $Q_{0,1}, Q_{0,2}\in S^0_{1,0}$ and  $\|\mathcal{R}_3u\|_2 \le \|u\|_2$, since $(\japa^{\theta-2})^{(3)}\in L^2(\R)$. 

On the other hand,
\begin{equation*}
\|(\japa^{\theta})'Q_{0,1}(\partial_x)J^{\frac12}u\|_2\le c \|Q_{0,1}(\partial_x)J^{\frac12}u\|_2\le  c\|J^{\frac12}u\|_2\le \|u\|_{\frac12,2}, 
\end{equation*}
and similarly
\begin{equation*}
\|(\japa^{\theta})''Q_{0,2}(\partial_x)J^{-\frac12}u\|_2\le \|u\|_2
\end{equation*}
since $(\japa^{\theta})'$ and $(\japa^{\theta})''\in L^{\infty}(\R)$.

Thus after integrating in the time interval $[\bar{t}_1,\bar{t}_2]$ we find that
\begin{equation*}
\int_{\bar{t}_1}^{\bar{t}_2} \int |\japa^{\alpha-2}\pop J^{\frac32} u|^2\, dxdt<\infty
\end{equation*}
which implies that there exists $t^{*}$ such that $\japa^{\alpha-2}\pop J^{\frac32} u(x, t^{*})\in L^2(\R)$.

\medskip

\subsection*{CLAIM} $\japa^{\alpha-2}J^2u(x, t^{*})\in L^2(\R)$.

\vspace{3mm}

\begin{proof} 

We write
\begin{equation}\label{weights-47}
\japa^{\alpha-2}J^{\frac12}J^{\frac32}u=-[J^{\frac12};\japa^{\alpha-2}]J^{\frac32}u+J^{\frac12}(\japa^{\alpha-2}J^{\frac32}u)
\end{equation}
Since $\japa^{\alpha-4}\in L^2(\R)$ the commutator \eqref{lem3-comm} implies that 
\begin{equation}
\|[J^{\frac12};\japa^{\alpha-2}]J^{\frac32}u\le c\|J u\|_2\le \|u\|_{1,2}.
\end{equation}
The last term in \eqref{weights-47} can be written as
\begin{equation*}
\|J^{\frac12}(\japa^{\alpha-2}J^{\frac32}u)\|_2 \le \|\japa^{\alpha-2}J^{\frac32}u\|_2 +\|\pop(\japa^{\alpha-2}J^{\frac32}u)\|_2,
\end{equation*}
so using that the $\japa^{\alpha-2}\in L^{\infty}(\R)$ the first term above is bounded by $\|u(t)\|_{\frac32,2}$
and one can write the last term like
\begin{equation*}
\pop(\japa^{\alpha-2}J^{\frac32}u)=\japa^{\alpha-2}\pop J^{\frac32}u) +[\pop; \japa^{\alpha-2}]J^{\frac32}u.
\end{equation*}
Then after integration in the time interval the first term is bounded because of \eqref{weights-47} and using once more
the commutator \eqref{lem3-comm} that
\begin{equation*}
\|[\pop; \japa^{\alpha-2}]J^{\frac32}u\|\le c\|u(t)\|_{1,2}.
\end{equation*}

Thus
\begin{equation}\label{star}
\japa^{\alpha-2} J^2u(x, t^{*})\in L^2(\R) \iff \japa^{\alpha}u(x, t^{*})\in L^2(\R).
\end{equation}

From the interpolation estimate \eqref{lem2b-interpolation}, one has that \eqref{star}
implies $J^{\alpha}u\in L^2(\R)$. Therefore, $u(x, t^{*})\in H^{\alpha}(\R)$. Thus our result follows.

\end{proof}

\vspace{.5cm}


\subsection*{Case $s\in [2,\;\frac52)$ with $\; 2\le s< \alpha<\frac52$, \;\;  $\alpha=2+\theta$} \hskip10pt

\vspace{3mm}

We know from the case $s\in (\frac32, 2)$, \hskip10pt $s<\alpha<2$ that since 
\begin{equation*}
|x|^{\alpha}u(\cdot, t_j)\in L^2(\R), \; j=1,2,\; t_1\neq t_2,
\end{equation*}

\begin{equation}\label{caso2-1}
\begin{cases}
\begin{aligned}
&u\in C([-T,T]: H^{2^{-}}(\R)),\hskip10pt |x|^{2^{-}}u\in L^{\infty}([-T,T]: L^{2}(\R)),\\
&\hskip5pt\text{and}\\
&\|\japa^{\theta_1}J^{\theta_2}u\|_{L^{\infty}_TL^2_x}<\infty \hskip5pt\text{and}\hskip5pt \theta_1+\theta_2<2.
\end{aligned}
\end{cases}
\end{equation}

We shall only consider the estimate for the nonlinear term.

\subsection*{Step 1} \hskip15pt

\begin{equation*}
\begin{split}
\Big|\int x\japa^{2\alpha-1} u\partial_x u u\,dx\Big|&\simeq \Big| -\frac13\int \japa^{2\alpha-1}u^3\,dx\Big|\\
&\le \|\japa^{2^{-}}u\|_2\|\japa^{2^{-}}u\|_2\|\japa^{\beta}u\|_{\infty}.
\end{split}
\end{equation*}

By using Sobolev
\begin{equation*}
\|\japa^{\beta}u\|_{\infty}\le \|J^{\frac12+}(\japa^{\beta}u)\|_2\le \|\japa^{\beta} J^{\frac12+}u\|_2
\end{equation*}
if $\beta+\frac12^{+}<2 \iff \beta<\frac32$.

So 
\begin{equation*}
\begin{split}
&2\alpha-1< 2^{-}+2^{-}+\beta< 4^{-}+\frac32 \iff\\
&2(2+\theta)-1< 4^{-}+\frac32 \iff 2\theta<\frac52 \iff \theta<\frac54.
\end{split}
\end{equation*}

\subsection*{Step 2} \hskip15pt

\vspace{3mm}

We also know (see \eqref{weights-3}, \eqref{weights-20}) that
\begin{equation*}\label{caso2-1b}
\begin{cases}
|x|^{\alpha}u\in L^{\infty}([t_1,t_2]: L^{2}(\R)),\\
|x|^{\alpha-\frac12} J^{\frac12}u\in L^2([t_1, t_2]: L^2(\R)).
\end{cases}
\end{equation*}

Consider ($2+\theta-\frac12=\frac32+\theta$)

\begin{equation*}
\begin{split}
A_2&=\int \japa^{\frac32+\theta} J^{\frac12}(u\partial_xu) \japa^{\frac32+\theta}J^{\frac12}u dx\\
&\simeq \int \japa^{\frac32+2\theta}J^{\frac12}(uu)\,J^{\frac12}\partial_x u dx\\
&=\int  \japa^{\frac32+2\theta-\frac12^{-}}J^{\frac12}(uu)\, \japa^{\frac12^{-}}J^{\frac12}\partial_x u dx\\
&\lesssim \|\japa^{\frac32+2\theta-\frac12^{-}}J^{\frac12}(uu)\|_2\|\japa^{\frac12^{-}}J^{\frac12}\partial_x u\|_2\\
&:= A_{21}(t)\cdot A_{22}(t).
\end{split}
\end{equation*}

From \eqref{caso2-1} it follows that

\begin{equation*}
\|A_{22}(\cdot)\|_{L^{\infty}_T}\lesssim \|\japa^{\frac12^{-}}J^{\frac32}u\|_{L^{\infty}_TL^2_x}<\infty.
\end{equation*}

On the other hand,

\begin{equation*}
A_{21}(t)\lesssim \|J^{\frac12}(\japa^{\frac32+2\theta-\frac12^{-}} u^2)\|_2+\|[J^{\frac12};\japa^{\frac32+2\theta-\frac12^{-}}](u^2)\|_2
\end{equation*}

The second term can be bounded by using formula \eqref{lem3-comm}. For the first term we have that

\begin{equation*}
\begin{split}
\|J^{\frac12}(\japa^{\frac32+2\theta-\frac12^{-}} u^2)\|_2 &\le \|J^{\frac12}(\japa^{\theta_1} u)\|_2
\|\japa^{\theta_2}u\|_{L^{\infty}}\\
&\le \|J^{\frac12}(\japa^{\theta_1} u)\|_2\|J^{\frac12^{+}}\japa^{\theta_2}u\|_2\\
&\simeq \|\japa^{\theta_1}J^{\frac12} u\|_2\|\japa^{\theta_2}J^{\frac12^{+}}u\|_2
\end{split}
\end{equation*}

Above $\theta_1=\alpha-\frac12$ and $\theta+\frac12^{+}<2\iff \theta_2<\frac32$. 

So
\begin{equation*}
\begin{split}
&\theta_1+\theta_2=\alpha-\frac12+\frac32^{-}>3+2\theta-\frac12^{-}\\
&\iff 2+\theta-\frac12+\frac32^{-}>\frac52+2\theta\iff \frac12>\theta.
\end{split}
\end{equation*}

\subsection*{Step 3} \hskip15pt

\vspace{3mm}

Since
\begin{equation*}
\int \japa^{\alpha-1}\partial_x (\partial_xu)\japa^{\alpha-1}\partial_xu dx\simeq \int \japa^{2\alpha-2} \partial_xu\partial_xu\partial_xu dx.
\end{equation*}
We get that 
\begin{equation*}
\begin{split}
\|\int \japa^{\alpha-1}\partial_x (\partial_xu)\japa^{\alpha-1}\partial_xu dx\|_{L^2_T}\lesssim \|\japa^{\alpha-1}\partial_xu\|_{L^2_TL^2_x}\|\partial_xu\|_{L^{\infty}_TL^2_x}
\end{split}
\end{equation*}
the last two terms are bounded because the previous step and the fact that $u\in C([-T,T]:H^s(\R))$, $s>\frac32$.

\subsection*{Step 4} \hskip15pt

\vspace{3mm}

\begin{equation*}
\begin{split}
&\int \japa^{\alpha-\frac32} J^{\frac32}(u\partial_xu)  \japa^{\alpha-\frac32} J^{\frac32}u\,dx\\
&\simeq \int J^{\frac32}\big(\japa^{\alpha-\frac32}u\partial_xu\big)\japa^{\alpha-\frac32} J^{\frac32}u\,dx
-\int [J^{\frac32}; \japa^{\alpha-\frac32}]\partial_x (u^2)\,dx\\
&:= A_{41}+A_{42}.
\end{split}
\end{equation*}

From the previous step
$$
\|\japa^{\alpha-\frac32} J^{\frac32}u\|_{L^2_TL^2_x}< \infty.
$$

By formula \eqref{lem3-comm} one gets
\begin{equation*}
[J^{\frac32}; \japa^{\alpha-\frac32}]\partial_x (u^2)\sim \japa^{\alpha-\frac52} J^{\frac32}(u^2)+\text{easy to deal terms}.
\end{equation*}
Thus if $\alpha<\frac52$, the first term on the right hand side is in $L^2$.
This handles the estimate of $A_{42}$

We rewrite $A_{41}$ as
\begin{equation*}
\begin{split}
A_{41}&= \int \japa^{\alpha-\frac32}u J^{\frac32}\partial_xu \japa^{\alpha-\frac32}J^{\frac32}u dx\\
&\hskip10pt \int [J^{\alpha-\frac32}; \japa^{\alpha-\frac32}u]\partial_xu \japa^{\alpha-\frac32} J^{\frac32}u\,dx\\
&:= A_{411}+A_{412}.
\end{split}
\end{equation*}

By integration by parts
\begin{equation*}
\begin{split}
A_{411}(t) &\simeq \int \japa^{2\alpha-3}\partial_x u J^{\frac32}uJ^{\frac32}u dx+ \text{terms we can control}\\
&\lesssim \|\partial_x u\|_{\infty}\|\japa^{2\alpha-3}J^{\frac32}u\|_2^2.
\end{split}
\end{equation*}
Thus \begin{equation*}
\int_{t_3}^{t_4} A_{411}(t)\,dt \lesssim \|\partial_x u\|_{L^{\infty}_TL^{\infty}_x}
\|\japa^{2\alpha-3}J^{\frac32}u\|_{L^2_TL^2_x}^2.
\end{equation*}
The last term is bounded by the previous step.

Consider now
\begin{equation*}
\begin{split}
|A_{412}(t)| &\le \|[J^{\alpha-\frac32}; \japa^{\alpha-\frac32}u]\partial_xu \|_2\|\japa^{\alpha-\frac32}u] J^{\frac32}u\|_2\\
&\le \|[J^{\frac32}; \japa^{\alpha-\frac32}u]\partial_xu \|_2^2+ \|\japa^{\alpha-\frac32}u] J^{\frac32}u\|_2^2.
\end{split}
\end{equation*}
The last term on the right hand side is controlled after integration on time by using the previous step.

Let $\alpha-\frac32=\frac12+\theta$. Using the commutator \eqref{comm-go-2} we get the following chain of inequalities
\begin{equation*}
\begin{split}
 \|[J^{\frac32}; \japa^{\frac12+\theta}u]\partial_xu \|_2
& \le c\Big(\|(\japa^{\frac12+\theta}u)'\|_{2^{+}}\|J^{\frac32-1}\partial_xu\|_{\infty^{-}}\\
&\hskip10pt+\|\partial_x u\|_{\infty}\|J^{\frac32}(\japa^{\frac12+\theta}u)\|_2\Big)
\end{split}
\end{equation*}

By applying Sobolev  it follows that
\begin{equation*}
\begin{split}
\|(\japa^{\frac12+\theta}u)'\|_{2^{+}}\le \|J^{\frac32}(\japa^{\frac12+\theta}u)\|_2\sim  \|\japa^{\frac12+\theta} J^{\frac32}u)\|_2
\end{split}
\end{equation*}
by previous step with coefficients in $L^{\infty}_T$

$$
\|\partial_x u\|_{L^{\infty}_TL^{\infty}_x} <\infty.
$$
and 
$$
\|J^{\frac32-1}\partial_xu\|_{\infty^{-}}\sim \|J^{\frac32}u\|_{\infty^{-}} \le \|J^{2^{-}}u\|_2
$$
uniformly, bounded.

\subsection*{Step 5} \hskip15pt

\vspace{3mm}

Integration by parts leads to
\begin{equation*}
\begin{split}
\int \japa^{\theta}\partial_x^2(u\partial_xu)\japa^{\theta}\partial_x^2u\sim \int \japa^{\theta}(\partial_x u\partial_x^2u)\japa^{\theta}\partial_x^2u +\text{terms easy
to control}
\end{split}
\end{equation*}
but
\begin{equation*}
\int \japa^{\theta}(\partial_xu\partial_x^2u)\japa^{\theta}\partial_x^2u\, dx  \le \|\partial_x u\|_{\infty}\|\japa^{\theta}\partial_x^2u\|_2^2,
\end{equation*}
the last term is bounded by previous step.

Now, for $0\le \theta\le\frac12$, $\alpha=2+\theta$, $2\le \alpha<\frac52$, we have
\begin{equation*}
\|\japa^{2+\theta}u(x, t_j)\|_2 <\infty,
\end{equation*}
which implies that
\begin{equation*}
\|\japa^{\theta-\frac12}J^{\frac52}u(x, \hat{t}\, )\|_2 <\infty 
\end{equation*}
(Conclusion of the step 5 at one time $\hat{t}$).

By interpolation,
$$
\|J^{\alpha}u\|_2<\infty.
$$
Thus we have that
\begin{equation*}
u\in C([-T,T]: H^{\alpha}(\R)).
\end{equation*}

\vspace{.5cm}

\section{Sketch of Proof of Theorem \ref{to-the-right}}

To help the reading of the following sections we start with two lemmas whose proofs can be found in \cite{MPS} which regard properties of the operators $\lop$ and $q(\partial_x)$ defined in Section 2.

\begin{lemma}\label{est-1}
Let $\phi$ be a smooth function. The operator $\mathcal{L}$ satisfies
\begin{equation}\label{ilw-e15}
\big([\mathcal{L}(\partial_x)\partial_x;\phi]\,f\big)(x)= c\phi'(x)(\Omega'(\partial_x)f)(x) + (R_1f)(x),
\end{equation}
with $c>0$ and 
\begin{equation}\label{ilw-e16}
\|R_1f\|_2\le c\|\phi''\|^{1/2}_2\|\phi'''\|_2^{1/2}\|f\|_2.
\end{equation}
\end{lemma}

\begin{proof} We use the argument in \cite{MPS} page 1040 and the proof of 
Lemma 2.4 in \cite{FLMP}.
\end{proof}

\begin{lemma}\label{est-3}
Let $\rho$ be a smooth function and consider the operator $q(\partial_x)$, then
it holds that
\begin{equation}\label{est-2}
\|q(\partial_x)\,[q(\partial_x); \rho]\,f\|_2\le c\| \rho'\|_{2,2} \|f\|_2.
\end{equation}
\end{lemma}

For a proof of Lemma  \eqref{est-1} see \cite{MPS} page 1040--1041. 

\begin{remark}
Following the argument in Lemma  2.4 in \cite{FLMP} one can write the following version of  estimate \eqref{est-2}:
\begin{equation}\label{est-2b}
\|q(\partial_x)\,[q(\partial_x); \rho]\,f\|_2\le c\,\| \rho''\|^{1/2}_2(\|\rho'\|_2^{1/2}+\|\rho'''\|^{1/2}_2)\|f\|_2.
\end{equation}
\end{remark}

\medskip

Next, we present a virial identity useful to establish Theorem \ref{to-the-right}.

\begin{lemma}\label{virial-right} Let $\varphi(x,t)\in L^{\infty}\cap C(\R: X)$ and $u\in C(\R, H^s(\R))$ be a solution of the IVP associated to equation \eqref{ILW}, then

\begin{equation}\label{virial-1}
\begin{split}
\frac{d}{dt}\int u^2(x,t) \varphi(x,t)\,dx&=-c\,\int \partial_x\varphi \,q(\partial_x)u \,q(\partial_x)u\,dx \\
&\hskip15pt- c\int u\,(q(\partial_x) \,[q(\partial_x); \partial_x\varphi]u)\,dx\\
&\hskip15pt - c\,\int u R_1(\cdot)u\,dx\\
&\hskip15pt+ \frac{2}{3}\int u^3\partial_x \varphi(x,t)\,dx+\int u^2\partial_t\varphi\,dx\\
& = E_1+E_2+E_3+E_4+E_5.
\end{split}
\end{equation}
where $R_1$ is given in Lemma \ref{est-1}.
\end{lemma}

\begin{proof}
\begin{equation}\label{virial-2}
\begin{split}
\frac{d}{dt}\int &u^2(x,t)\varphi(x,t)\,dx= 2\int u \partial_t u\varphi \,dx +\int u^2\partial_t\varphi\,dx\\
&=2\int u\big( \lop\partial_xu\big)\varphi\,dx -2\int u^2\partial_x u\varphi\,dx +\int u^2\partial_t\varphi\,dx\\
&=2\int u\big( \lop\partial_xu\big)\varphi\,dx+\frac23 \int u^3\partial_x\varphi\,dx +\int u^2\partial_t\varphi\,dx\\
&= 2\int u\big( \lop\partial_xu\big)\varphi\,dx+E_4+E_5.
\end{split}
\end{equation}

Using that the symbol $\lop(\xi)$ is even we have
\begin{equation}\label{virial-3}
\begin{split}
2\int u\varphi\big( \lop\partial_xu\big)\,dx&=
2\int  \lop(u \varphi)\partial_xu \,dx\\
&=-2\int \lop\partial_x(u \varphi) u \,dx\\
&= -2\int u\varphi   \lop\partial_xu\,dx\\
&\hskip12pt-2\int u\, [\lop\partial_x; \varphi] u\,dx.
\end{split}
\end{equation}

Thus 
\begin{equation}\label{virial-4}
\int u\varphi\,\big( \lop\partial_xu\big)\,dx=-\frac12 \int u\, [\lop\partial_x; \varphi] u\,dx.
\end{equation}

Using \eqref{ilw-e15} we obtain
\begin{equation}\label{virial-5}
\begin{split}
&\int u \varphi\,\big( \lop\partial_xu\big)\,dx\\
&=
-c\,\int u\partial_x\varphi \,\Omega_{\delta}'(\partial_x)u\,dx - c\,\int u R_1(\cdot)u\,dx\\
&= -c\,\int u\partial_x\varphi \,q(\partial_x)q(\partial_x)u\,dx - c\,\int u R_1(\cdot)u\,dx\\
&= -c\,\int q(\partial_x)(u\partial_x\varphi) \,q(\partial_x)u\,dx - c\,\int u R_1(\cdot)u\,dx\\
&=- c\,\int \partial_x\varphi \,q(\partial_x)u \,q(\partial_x)u\,dx 
- c\int u\,(q(\partial_x) \,[q(\partial_x); \partial_x\varphi]u)\,dx\\
&\hskip12pt- c\,\int u R_1(\cdot)u\,dx.
\end{split}
\end{equation}

Combining \eqref{virial-2} and \eqref{virial-5}, we complete the proof of the lemma.

\end{proof}

\vspace{1cm}

\begin{proof}[Proof of Theorem \ref{to-the-right}]

Choosing 
\begin{equation}\label{toright-1}
\varphi(x,t)=\chi\Big(\frac{x-c_1}{c_0 t}\Big).
\end{equation}
as in the proof of Theorem 1.5 in \cite{MPS} with $\chi$ satisfying
\begin{equation}
\label{chi}
\begin{cases}
\chi\in C^\infty(\R), \quad 0\leq \chi \leq 1 \quad \hbox{ in}\quad \R, \\
\chi(s) \equiv 0 \hskip34pt \hbox{if}\hskip 5pt s\leq 1, \quad  \chi(s) \equiv 1 \quad \hbox{if}\quad s\geq 2, \\
\chi' (s) >0, \hskip26pt \hbox{in}\hskip 5pt (1,2),\\
|\chi^{(k)} (s)|\leq 2^k, \hskip5pt \hbox{for any}\hskip 5pt s\in \R, \quad k=1,2,3,
\end{cases}
\end{equation}
and $c_0, c_1$ constants to be chosen latter. We deduce that
 
\begin{equation}\label{toright-2}
\partial_t \Big(\chi\Big(\frac{x-c_1}{c_0 t}\Big)\Big)=\chi'\Big(\frac{x-c_1}{c_0 t}\Big)\Big(\frac{x-c_1}{c_0 t}\Big)\frac{-1}{t}\le \chi'(\cdot)\Big(\frac{-1}{t}\Big)
\end{equation}
and
\begin{equation}\label{toright-3}
\partial_x \chi\Big(\frac{x-c_1}{c_0 t}\Big)=\chi'(\cdot)\frac{1}{c_0t}.
\end{equation}

\medskip

Using Lemma \ref{virial-right}, Lemma \ref{est-1} and Lemma \ref{est-2} we deduce that

\begin{equation}\label{toright-4}
\begin{split}
E_1&\le 0,\\
|E_2| &\le c_{\chi} \Big(\frac{1}{c_0t}\Big)^{\frac32}\Big(1+\frac{1}{c_0t}\Big)\|u(t)\|_2^2\le c_{\chi} \Big(\frac{1}{c_0t}\Big)^{\frac32}
\|u(t)\|_2^2,\\
|E_3|&\le c_{\chi}\Big(\frac{1}{c_0t}\Big)^{\frac52}\|u(t)\|_2^2\le  c_{\chi}\Big(\frac{1}{c_0t}\Big)^{\frac32}\|u(t)\|_2^2,\\
|E_4|&\le \frac{2\|u(t)\|_{\infty}}{3c_0 t}\int u^2(x,t)\chi'\Big(\frac{x-c_1}{c_0 t}\Big)\,dx,
\end{split}
\end{equation}
and 
\begin{equation}\label{toright-5}
E_5 \le \frac{-1}{t}\int u^2(x,t)\, \chi'\Big(\frac{x-c_1}{c_0 t}\Big)\,dx
\end{equation}

Now we use the above estimates, the identity \eqref{virial-1}, Sobolev embedding and the $L^2$ conserved quantity of the
ILW equation to yield the inequality
\begin{equation}\label{toright-6}
\begin{split}
\frac{d}{dt}\int &u^2(x,t) \chi\Big(\frac{x-c_1}{c_0 t}\Big)\,dx\\ 
&\le \frac{c_{\chi}}{(c_0t)^{\frac32}} \|u_0\|_2^2+
\left(\frac{2c_2\|u_0\|_{1,2}}{3c_0t}-\frac{1}{t}\right)\int u^2(x,t) \chi'\Big(\frac{x-c_1}{c_0 t}\Big)\,dx
\end{split}
\end{equation}
where $c_2$ is a universal constant. Hence taking $c_0$ such that 
\begin{equation*}
\frac{2c_2\|u_0\|_{1,2}}{3c_0}<1
\end{equation*}
and for any $\epsilon>0$,  we fix $t_1>0$, so that
\begin{equation*}
 \|u_0\|_2^2\int_{t_1}^{\infty}\frac{c_{\chi}}{(c_0t)^{\frac32}}\,dt \le \epsilon.
 \end{equation*}

 \medskip
 
 Integrating \eqref{toright-6} in the time interval $[t_1,t_2]$ it transpires that 
 \begin{equation*}
 \int u^2(x,t_2) \chi\Big(\frac{x-c_1}{c_0 t}\Big)\,dx\le  \int u^2(x,t_1) \chi\Big(\frac{x-c_1}{c_0 t}\Big)\,dx +\epsilon.
 \end{equation*}
 
 Next we fix $c_1$  such that
 \begin{equation*}
 \int u^2(x,t_1) \chi\Big(\frac{x-c_1}{c_0 t}\Big)\,dx\le \epsilon,
 \end{equation*}
and thus for any $t_2>t_1$ we get
\begin{equation*}
 \int_{x>c_1+2c_0t} u^2(x,t_2)\,dx\le  \int u^2(x,t_1) \chi\Big(\frac{x-c_1}{c_0 t}\Big)\,dx +\epsilon\le 2\epsilon.
 \end{equation*}
 
 To end the proof we fix $C_0= 3c_0$ to yield
 
 \begin{equation*}
 \underset{t\to\infty}{\limsup} \int_{x>C_0t} u^2(x,t_2)\,dx\le 2\epsilon.
 \end{equation*}
 
 \end{proof}

\section{Sketch Proof of Theorem \ref{L2ILW}}

\subsection{Auxiliary Notation} We start by introducing the notation need it in our arguments to prove 
 Theorem \ref{L2ILW} and \ref{energy-ILW}.

 Let $\phi$ be a smooth even and positive function such that 
\begin{equation}\label{weight}
\begin{cases}
\text{\;\;\rm i)}&\!\! \phi'(x)\leq 0, \;\text{for}\; x\ge 0, \\
\text{\;\rm ii)}&\!\! \phi(x)\equiv 1, \;\text{for}\; 0 \le x \le 1, \phi(x) = e^{-x}  \;\text{for}\; x\ge 2, \\
&\text{and}\;\; e^{-x} \leq \phi(x) \leq 3e^{-x}  \;\text{for}\; x\ge 0,\\
\text{\rm iii)}&\!\! |\phi'(x)| \leq c\,\phi(x) \;\text{and}\; |\phi''(x)| \leq c\,\phi(x)\\
 &\;\text{for some positive constant} \;c.
\end{cases}
\end{equation}

Let $\psi(x) = \displaystyle\int_0^x \phi (s)ds$. In particular, $|\psi (x)| \leq 1 + 3\displaystyle\int_1^{\infty}e^{-t}dt < \infty$.   \\

Next, we consider a smooth cut-off function $\zeta :\mathbb{R} \rightarrow \mathbb{R}$ such that 
\begin{equation}
 \zeta \equiv 1 \text{\hskip3pt in \hskip3pt} [0,1], {\hskip7pt} 0 \leq  \zeta\leq 1 \text{\hskip10pt and \hskip10pt} \zeta \equiv 0 \text{\hskip3pt in \hskip3pt} (\infty,-1]\cup[2,\infty),
\end{equation}
and define $\zeta_n(x) := \zeta(x-n)$.
 
For the parameters $\lambda, \sigma \in \mathbb{R}^+$,  we define 
\begin{equation*}
\phi_{\lambda}= \lambda\phi\left( \frac{x}{\lambda}\right) {\hskip 10 pt \text{and} \hskip10pt}  \psi_{\sigma}(x) = \sigma\psi\left( \frac{x}{\sigma}\right).
\end{equation*}

First, we start by considering some useful parameters involved in our argument of proof. 
 \begin{equation}\label{parameters}
 \rho(t) = \pm t^m, {\hskip5pt  \hskip5pt}  \mu_1(t) = \frac{t^b}{\log t} {\hskip10pt \hbox{and} \hskip10pt} \mu(t) = t^{(1-b)}\log^2t,
 \end{equation}
where $m$ and $b$ are positive constants satisfying the relations 
\begin{equation}\label{bmrelation}
  0 \leq m \leq 1 - \frac{b}{2}{\hskip10pt \hbox{and} \hskip10pt} 0< b \leq \min\left\{\frac{2}{3}, \frac{2}{2+q}\right\},\;\;\;\;\;\;q>1.
\end{equation}

Since
\begin{equation*}
  \frac{\mu'_1(t)}{\mu_1(t)} = \frac{b}{t} - \frac{1}{t\log t}{\hskip10pt \text{and} \hskip10pt} \frac{\mu'(t)}{\mu(t)} = \frac{(1-b)}{t} + \frac{2}{t\log t}
\end{equation*}
it readily follows that
    \begin{equation}\label{Relationderivatives}
        \frac{\mu'_1(t)}{\mu_1(t)} \sim \frac{\mu'(t)}{\mu(t)} = O\left(\frac1t\right), {\hskip10pt for \hskip10pt} t \gg 1
    \end{equation}
where $t \gg 1$ means the values of t such that $\mu_1'(t)$ is positive. In particular, $[10,+\infty) \subset \{t \gg 1\}$.

\medskip

For $u=u(x,t)$ a solution of the IVP  associated to \eqref{ILW} we define the functional\\ 
\begin{equation}\label{NC-1}
     \mathcal{I}(t) := \frac{1}{\mu(t)}\int_{\mathbb{R}} u(x,t)\psi_{\sigma}\left( \frac{x}{\mu_1(t)} \right) \phi_{\lambda }\left( \frac{x}{\mu_1^q(t)} \right)\,dx, 
        \end{equation}
for $q > 1$.

\begin{lemma}\label{Functionalbounded}
Let $u(\cdot,t) \in L^2(\mathbb{R})$, $t\gg 1$. The functional $\mathcal{I}(t)$  is well defined and bounded in time.
   
\begin{proof}
 The Cauchy-Schwarz inequality and the definition of the functions $\mu(t)$ and $\mu_1(t)$ imply that 
 \begin{equation}
  \begin{split}
   |\mathcal{I}(t)| &\leq \frac{1}{\mu(t)}\|u(t)\|_{L^2}\left\|\psi_{\sigma}\left( \frac{\cdot}{\mu_1(t)} \right)\right\|_{L^{\infty}}\left\|\phi_{\lambda}\left( \frac{\cdot}{\mu_1^q(t)}\right)\right\|_{L^2} \\
   & = \frac{\mu_1^{q/2}(t)}{\mu(t)}\|u(t)\|_{L^2}\|\psi_{\sigma}\|_{L^{\infty}}\|\phi_{\lambda}\|_{L^2}\\
   & \lesssim_{\sigma, \lambda} \frac{1}{t^{(2-2b-bq)/2}}\frac{1}{\log^{(4+q)/2}(t)}\|u_0\|_{L^2}. \\
  \end{split}
 \end{equation}

Since $b$ satisfies the condition \eqref{bmrelation} we have that
\begin{equation*}
\sup_{t \gg 1}|\mathcal{I}(t)| < \infty.
\end{equation*}   
\end{proof}
\end{lemma}

 \begin{lemma}\label{BoundL1}
For any $t \gg 1$,  it holds that
\begin{equation}\label{NC-2}
\frac{1}{\mu_1(t)\mu(t)} \int_{\mathbb{R}}u^2(x,t)\,\psi_{\sigma}'\Big( \frac{x}{\mu_1(t)} \Big)\phi_{\lambda}\Big( \frac{x}{\mu_1^q(t)} \Big)\,dx \leq 4\frac{d}{dt}\mathcal{I}(t) + h(t), 
\end{equation}
where $h(t)\in L^1(t \gg 1)$.
\end{lemma}
  
\begin{proof}
We have that
  \begin{equation}\label{l42-1}
   \begin{split}
   \frac{d}{dt}\mathcal{I}(t)
   &= \frac{1}{\mu(t)}\int_{\mathbb{R}}\partial_t\left( u\psi_{\sigma}\left( \frac{x}{\mu_1(t)} \right) \phi_{\lambda}\left( \frac{x}{\mu_1^q(t)} \right)\right)\,dx\\
   &\hskip10pt - \frac{\mu'(t)}{\mu^2(t)}\int_{\mathbb{R}} u\psi_{\sigma}\left( \frac{x}{\mu_1(t)} \right) \phi_{\lambda}\left( \frac{x}{\mu_1^q(t)} \right)\,dx\\
   &= A(t) + B(t).
   \end{split}
  \end{equation}
  
The Cauchy-Schwarz inequality and the conservation of mass, $I_2$ in \eqref{CL}, yield
    \begin{equation}\label{l42-2}
     \begin{aligned}
   |B(t)| &\leq \left|\frac{\mu'(t)}{\mu^2(t)}\right| \left\|u(t)\right\|_{L^2}\Big\|\psi_{\sigma}\Big( \frac{\cdot}{\mu_1(t)} \Big)\Big\|_{L^{\infty}}\Big\|\phi_{\lambda}\Big( \frac{\cdot}{\mu_1^q(t)} \Big)\Big\|_{L^2} \\
   & \lesssim_{\sigma , \lambda} \frac{1}{t^{(4-2b-bq)/2}}\frac{1}{\log^{(4+q)/2}t}\|u_0\|_{L^2}.\\
   \end{aligned}
    \end{equation}
Hence $B(t)\in L^1(\{t \gg 1\})$ whenever $b\leq \frac{2}{2 +q}$. We remark that this term is bounded in $\{t \gg 1\}$.


To estimate $A(t)$, we first differentiate in time to write
\begin{equation}\label{TermsEq.}
\begin{split}
A(t)&=  \frac{1}{\mu(t)}\int_{\mathbb{R}} u_t(x,t)\psi_{\sigma}\Big( \frac{{x}}{\mu_1(t)} \Big) \phi_{\lambda}\Big( \frac{{x}}{\mu_1^q(t)} \Big)\,dx\\
 &\hskip10pt - \frac{\mu_1'(t)}{\mu_1(t)\mu(t)}\int_{\mathbb{R}} u(x,t) \Big( \frac{{x}}{\mu_1(t)} \Big)\psi_{\sigma}'\Big( \frac{{x}}{\mu_1(t)} \Big) \phi_{\lambda}\Big( \frac{{x}}{\mu_1^q(t)} \Big)\,dx\\
&\hskip10pt  - \frac{q\mu_1'(t)}{\mu_1(t)\mu(t)}\int_{\mathbb{R}} u(x,t) \psi_{\sigma}\Big( \frac{{x}}{\mu_1(t)} \Big)\Big( \frac{{x}}{\mu_1^q(t)} \Big) \phi_{\lambda}'\Big( \frac{{x}}{\mu_1^q(t)} \Big)\,dx\\
& = A_1(t)+A_2(t)+A_3(t).
\end{split}
\end{equation}

Using the equation \eqref{ILW} and integrating by parts yield
\begin{equation} \label{l42-3}
\begin{split}
A_1(t) &=\frac{1}{\mu(t)}\int_{\mathbb{R}}\lop\partial_x u(x,t)\,\left(\psi_{\sigma}\left(\frac{{x}}{\mu_1(t)} \right) \phi_{\lambda}\left(\frac{{x}}{\mu_1^q(t)} \right)\right)\,dx\\
&\hskip15pt + \frac{1}{2\mu(t)\mu_1(t)}\int_{\mathbb{R}} u^2(x,t)\psi_{\sigma}'\left(\frac{{x}}{\mu_1(t)} \right) \phi_{\lambda}\left(\frac{{x}}{\mu_1^q(t)} \right)\,dx\\
    &\hskip15pt +\frac{1}{2\mu(t)\mu_1^q(t)}\int_{\mathbb{R}}u^2(x,t) \psi_{\sigma}\left(\frac{{x}}{\mu_1(t)} \right) \phi_{\lambda}'\left(\frac{{x}}{\mu_1^q(t)} \right)\,dx\\
    &=: A_{1,1}(t) + A_{1,2}(t) + A_{1, 3}(t).
    \end{split}
\end{equation}

We remark that  $ \, A_{1, 2}(t)$ is the term we want to estimate in \eqref{l42-3}. Then we need to show that the reminder terms
are in $L^1(\{t\gg 1\})$.

We notice that $\partial_x\lop= \partial_x^2\Psi(\partial_x)$ where the symbol of $i\Psi(\cdot)$, 
$i\Psi(\xi)=\frac{1}{\xi}-\coth(\xi)$ belongs to $C(\R)\cap L^{\infty}(\R)$, odd, real value, and order zero. 
Using this and integrating by parts we obtain that
\begin{equation}\label{l42-3b}
 A_{1,1}(t)= \frac{1}{\mu(t)}\int_{\mathbb{R}}\Psi(\partial_x)u(x,t)\,\partial_x^2\Big(\psi_{\sigma}\Big(\frac{{x}}{\mu_1(t)} \Big) \phi_{\lambda}\Big(\frac{{x}}{\mu_1^q(t)} \Big)\Big)\,dx.
 \end{equation}

 Differentiating with respect to $x$, using the Cauchy-Schwarz inequality, the fact that the symbol of $i\Psi_{\delta}(\partial_x)$ is of order zero, the conservation of mass, and the definition of $\mu(t)$ and $\mu_1(t)$, we deduce that  
\begin{equation}\label{l42-4}
\begin{split}
|A_{1,1}(t)|&\le\frac{1}{\mu(t)\mu_1^{3/2}(t)}\|u(t)\|_{L^2}\|\psi_{\sigma}''\|_{L^2}\|\phi_{\lambda}\|_{L^{\infty}}\\
&\hskip15pt +\frac{1}{\mu(t)\mu_1^{(1+2q)/2}(t)}\|u(t)\|_{L^2}\|\psi_{\sigma}'\|_{L^2}\|\phi_{\lambda}'\|_{L^{\infty}}\\
&\hskip15pt+\frac{1}{\mu(t)\mu_1^{3q/2}(t)}\|u(t)\|_{L^2}\|\psi_{\sigma}\|_{L^{\infty}}\|\phi_{\lambda}''\|_{L^2}\\
&\lesssim_{\sigma,\lambda}  \frac{ \|u_0\|_{L^2}}{t^{(2 + b)/2}\log^{1/2}t}+\frac{ \|u_0\|_{L^2}}{t^{(2 - b + 2bq)/2}\log^{(\frac{3}{2}-q)}(t)} \\
&\hskip15pt +\frac{ \|u_0\|_{L^2}}{t^{(2 - 2b + 3bq)/2}\log^{(4-3q)/2}t}.
\end{split} 
\end{equation}

Since $q>1$ it follows  that $A_{1,1}(t)\in L^1(\{t\gg1\})$.

The term $A_{1,3}$ can be bounded by employing the conservation of mass, and the definition of 
$\mu(t)$ and $\mu_1(t)$.
\begin{equation}\label{l42-9}
\begin{split}
 |A_{1,3}| &\leq  \frac{\|u(t)\|_{L^2}^2}{2|\mu(t)\mu_1^q(t)|}\left\|\psi_{\sigma}\left(\frac{{x}}{\mu_1(t)} \right)\right\|_{L^{\infty}}\left\|\phi_{\lambda}'\left(\frac{{x}}{\mu_1^q(t)} \right)\right\|_{L^{\infty}}\\ 
      &\lesssim_{\sigma,\lambda} \frac{\|u_0\|_{L^2}^2}{t^{1-b + bq}\log^{(2-q)}t},
\end{split}
\end{equation}
because of  $q > 1$ one has that $A_{1,3}(t) \in L^1(\{t \gg1\})$.
               
\medskip

Next we turn our attention to the other terms of (\ref{TermsEq.}). First, by means of Young's inequality, we have for $\epsilon > 0$
\begin{equation*}
\begin{split}
 |A_2(t)|
  &\leq \Big|\frac{\mu_1'(t)}{\mu_1(t)\mu(t)} \Big|\int_{\mathbb{R}}\Big|\psi_{\sigma}'\Big( \frac{{x}}{\mu_1(t)} \Big) \phi_{\lambda}\Big( \frac{{x}}{\mu_1^q(t)} \Big)\Big|\Big[ \frac{u^2}{4\epsilon} + 4\epsilon\Big|\frac{{x}}{\mu_1(t)}\Big|^2\Big]\,dx    \\
  &\leq \frac{1}{4\epsilon}\Big|\frac{\mu_1'(t)}{\mu_1(t)\mu(t)} \Big|\int_{\mathbb{R}} u^2(x,t)\Big|\psi_{\sigma}'\Big( \frac{{x}}{\mu_1(t)} \Big) \phi_{\lambda}\Big( \frac{{x}}{\mu_1^q(t)} \Big)\Big|\,dx\\
  &\hskip15pt + 4\epsilon\Big|\frac{\mu_1'(t)}{\mu_1(t)\mu(t)} \Big| \,\Big\|\phi_{\lambda}\Big( \frac{\cdot}{\mu_1^q(t)} \Big)\Big\|_{L^{\infty}}
 \Big \|\Big(\frac{\cdot}{\mu_1(t)}\Big)^2\psi_{\sigma}'\Big(\frac{\cdot}{\mu_1(t)}\Big)\Big\|_{L^1}.
    \end{split}
    \end{equation*}

Then, taking $\epsilon = \mu_1'(t)$, which is positive in $\{t\gg 1 \} $, we get
\begin{equation}\label{l42-10}
\begin{split}
 |A_2(t)| & \leq \Big|\frac{1}{4\mu_1(t)\mu(t)} \Big|\int_{\mathbb{R}}u^2\psi_{\sigma}'\Big( \frac{{x}}{\mu_1(t)}\Big) \phi_{\lambda}\Big( \frac{{x}}{\mu_1^q(t)} \Big)dx \\
        &\hskip15pt+C_{\sigma,\lambda}\frac{(b\log t-1)^2}{t^{3-3b}\log^6t}\\
        &= \frac{1}{2}A_{1,2}(t) +C_{\sigma, \lambda}\frac{(b\log t-1)^2}{t^{3-3b}\log^6t},
     \end{split}
    \end{equation}
where $C_{\sigma,\lambda}$ is a constant depending on $\sigma$ and $\lambda$.

Notice that the last term in the last inequality of \eqref{l42-10} is integrable in ${t \gg 1}$ since $b < \frac{2}{3}$.

Finally, we consider the term $A_3(t)$. Young's inequality and the conservation of mass tell us that
    \begin{equation}\label{l42-11}
    \begin{split}
     |A_3(t)|&\leq \left| \frac{q\mu_1'(t)}{\mu_1(t)\mu(t)}\right|\|\psi_{\sigma}\|_{L^{\infty}}\int_{\mathbb{R}} t^{1-b} u^2(x,t)\,dx\\
      &\hskip15pt +\left| \frac{q\mu_1'(t)}{\mu_1(t)\mu(t)}\right| \|\psi_{\sigma}\|_{L^{\infty}}\!\!\int_{\mathbb{R}}\frac{1}{t^{1-b}}\left[\left( \frac{{x}}{\mu_1^q(t)} \right) \phi_{\lambda}'\left( \frac{x}{\mu_1^q(t)} \right)\right]^2\,dx\\
     & \lesssim_{\sigma,\lambda} \left| \frac{qt^{1-b}\mu_1'(t)}{\mu_1(t)\mu(t)}\right| + \left| \frac{q\mu_1'(t)\mu_1^q(t)}{t^{1-b}\mu_1(t)\mu(t)}\right|.
    \end{split}
    \end{equation}
    
    Hence, the conditions on \eqref{Relationderivatives} imply 
\begin{equation}\label{l42-12}
|A_{3}(t)| \lesssim_{\sigma,\lambda} \frac{1}{t\log^2t} + \frac{1}{t^{3-b(2+q)}\log^{2+q}t}.
\end{equation}
Since $b \leq \frac{2}{2+q} $, $A_3(t) \in L^1(\{t \gg 1\})$.

Gathering the information in \eqref{l42-1}, \eqref{TermsEq.}, \eqref{l42-4}, \eqref{l42-9}, \eqref{l42-10}, \eqref{l42-11} 
and \eqref{l42-12} together we conclude that
    \begin{equation}\label{l42-13}
    \begin{split}
\frac{1}{2\mu(t)\mu_1(t)}\int_{\mathbb{R}} u^2(x,t)\psi_{\sigma}'\left(\frac{{x}}{\mu_1(t)} \right) \phi_{\lambda}\left(\frac{{x}}{\mu_1^q(t)} \right)\,dx\leq & \frac{d}{dt}\mathcal{I}(t)+h(t)
    \end{split}
    \end{equation}
where $h(t)\in L^1(\{t\gg 1\})$, as desired.
\end{proof}

The next lemma will give us a key bound in our analysis.

\begin{lemma}\label{intl2bounded}
 Assume that $u_0 \in L^2(\R)$. Let $u \in C(\mathbb{R}: L^2(\mathbb{R}))\cap L^{\infty}(\mathbb{R}: L^2(\mathbb{R})) $ be the solution of the IVP associated to \eqref{ILW} given in \cite{IfSa}. Then, there exists a constant $0< C < \infty$, such that 
 \begin{equation}
  \int_{ \{t \gg 1\} } \frac{1}{t\log t}\int_{B_{t^b}} u^2(x,t)\, dxdt \leq C. 
 \end{equation}

\end{lemma}

\begin{proof}
 From the definition, $\mu(t)\mu_1(t)=t\log t$ and a straightforward computation involving the properties of the function $\phi$, it follows that
    \begin{equation*}
    \frac{1}{\mu_1(t)\mu(t)}\int_{B_{t^b}}u^2(x,t)\,dx \leq \frac{1}{\mu_1(t)\mu(t)}\int_{\mathbb{R}}u^2\psi_{\sigma}'\left( \frac{{x}}{\mu_1(t)}\right)\phi_{\lambda}\left( \frac{{x}}{\mu_1^q(t)}\right) \,dx, 
    \end{equation*}
for suitable $\sigma$ and $\delta$, whenever $q> 1$ is chosen sufficiently close to 1 and $b$ slightly smaller if necessary. Lemma \ref{BoundL1} implies that
\begin{equation}\label{l43-1}
 \int_{\{t \gg 1\}}\frac{1}{\mu_1(t)\mu(t)}\int_{B_{t^b}}u^2(x,t)\,dxdt \leq -\mathcal{I}(t) + \int_{\{t \gg 1\}}|h(t)|\,dt.
\end{equation}
The first term on the right hand side of inequality \eqref{l43-1} is bounded because of  $b \leq \frac{2}{2 + q} < \frac{2}{3}$ and the last one is bounded by the proof of Lemma \ref{BoundL1}. This
completes the proof of the lemma. 
 
\end{proof}

Now we are ready to prove Theorem \ref{L2ILW}.
        
\subsection{Proof of Theorem  \ref{L2ILW}}

Since the function $\frac{1}{t\log t} \notin L^1(B^c_r(1))$, from  the previous lemma, we can ensure that there exists a sequence $(t_n) \to \infty$, such that 
\begin{equation*}
 \lim_{n \to \infty}\int_{B_{(t_n)}^b}u^2(x, t_n)\,dx = 0.
\end{equation*}
Therefore, $0$ is an accumulation point and using that $u^2 \geq 0$ we can conclude the result.

\section{Sketch Proof of Theorem \ref{energy-ILW}}

\begin{lemma}\label{ILWasymp}
Let u $\in C(\mathbb{R}: H^{\frac{1}{2}}(\mathbb{R}))\cap L^{\infty}(\mathbb{R}: H^{\frac{1}{2}}(\mathbb{R}))$ the solution
of the IVP associated to \eqref{ILW} given in \cite{IfSa}. Then, there exists a constant $C> 0$ such that
\begin{equation}\label{l61-2}
 \int_{\{ t \gg 1 \}}\frac{1}{t\log t}\int_{B_{t^b}(0)} | q(\partial_x)u(x,t)|^2\,dxdt \leq C,
\end{equation}
where $q(\partial_x)$ is defined in \eqref{a7}.
\end{lemma}

\begin{proof}
 Consider the functional 
    \begin{equation}\label{l61-3}
{\mathcal{J}}(t) := \frac{1}{\mu(t)}\int_{\mathbb{R}}u^2(x,t)\psi_{\sigma}\left(\frac{x}{\mu_1(t)} \right)\,  dx.
    \end{equation}
 where $\mu(t)$ and $\mu_1(t)$ were defined in \eqref{parameters}.

Differentiating \eqref{l61-3} yields
\begin{equation}\label{l61-4}
\begin{split}
     \frac{d}{dt}{\mathcal{J}}(t) &=- \frac{\mu'(t)}{\mu^2(t)}\int_{\mathbb{R}}u^2(x,t)\,\psi_{\sigma}\left(\frac{x}{\mu_1(t)} \right)  \,dx\\
      &\hskip10pt +\frac{2}{\mu(t)}\int_{\mathbb{R}}u(x,t)\partial_tu(x,t)\,\psi_{\sigma}\left(\frac{x}{\mu_1(t)} \right)  \, dx\\
     &\hskip10pt - \frac{\mu_1'(t)}{\mu(t)\mu_1(t)}\int_{\mathbb{R}}u^2(x,t)\,\phi_{\sigma}\left(\frac{x}{\mu_1(t)} \right)\left(\frac{x}{\mu_1(t)} \right)\, dx\\
     &= A(t)+B(t)+C(t).
\end{split}
\end{equation}
    
Combining the properties $\mu(t)$ and $\mu_1(t)$, the conservation of mass, and using \eqref{parameters}, it follows that
    \begin{equation}\label{l61-5}
    |A(t)|+|C(t)| \lesssim_{\sigma} \frac{\|u_0\|_{L^2}^2}{t^{2-b}\log^2t}.
    \end{equation}
Thus, the terms $A(t)$, $C(t)$ are integrable in $\{t \gg 1\}$.


Regarding $B(t)$, we use the equation in \eqref{ILW} and integrate by parts to write
\begin{equation}\label{l61-7a}
\begin{split}
B(t) &=  \frac{2}{\mu(t)}\int_{\mathbb{R}}u\,\big(\lop\partial_x u\big)\,\psi_{\sigma}\Big(\frac{x}{\mu_1(t)} \Big)  dx \\
&\hskip15pt + \frac{2}{3\mu(t)\mu_1(t)}\int_{\mathbb{R}}u^3\phi_{\sigma}\Big(\frac{x}{\mu_1(t)} \Big)  dx \\
&= B_1(t) + B_2(t).
\end{split}
\end{equation}
    
Using that the symbol of the operator $\lop$ is even and integration by parts it follows that 
that
    
\begin{equation}\label{l61-7b}
\begin{split}
B_1(t)  &=  \frac{2}{\mu(t)}\int \lop\Big(u \psi_{\sigma}\Big(\frac{x}{\mu_1(t)} \Big)\Big)\,\partial_x u\, dx\\
&= -\frac{2}{\mu(t)}\int \lop\partial_x\Big(u \psi_{\sigma}\Big(\frac{x}{\mu_1(t)} \Big)\,u\, dx\\
&=-\frac{2}{\mu(t)}\int u \psi_{\sigma}\Big(\frac{x}{\mu_1(t)} \Big) \lop\partial_xu\,dx\\
&\hskip15pt-\frac{2}{\mu(t)}\int u\,[\lop\partial_x; \psi_{\sigma}] u\,dx.
\end{split}
\end{equation}

Hence
\begin{equation}\label{l61-7c}
\begin{split}
B_1(t) &= -\frac{1}{\mu(t)}\int u\,[\lop\partial_x; \psi_{\sigma}] u\,dx.
\end{split}
\end{equation}

From Lemma \ref{est-1}

\begin{equation}\label{l61-7}
\begin{split}
B_1(t) &=-\frac{c}{\mu(t)\mu_1(t)}\,\int u  \phi_{\sigma}\Big(\frac{x}{\mu_1(t)} \Big) \,\Omega_{\delta}'(\partial_x)u\,dx\\ 
&\hskip15pt-\frac{c}{\mu(t)}\,\int u R_1(\cdot)u\,dx\\
&= -\frac{c}{\mu(t)\mu_1(t)}\,\int u\phi_{\sigma}\Big(\frac{x}{\mu_1(t)} \Big)\,q(\partial_x)q(\partial_x)u\,dx \\
&\hskip15pt- \frac{c}{\mu(t)}\,\int u R_1(\cdot)u\,dx\\
&= -\frac{c}{\mu(t)\mu_1(t)}\,\int q(\partial_x)\Big(u \phi_{\sigma}\Big(\frac{x}{\mu_1(t)} \Big)\Big)\,q(\partial_x)u\,dx \\
&\hskip15pt-\frac{c}{\mu(t)}\,\int u R_1(\cdot)u\,dx\\
&=- \frac{c}{\mu(t)\mu_1(t)}\,\int  \phi_{\sigma}\Big(\frac{x}{\mu_1(t)} \Big)\,q(\partial_x)u \,q(\partial_x)u\,dx\\ 
&\hskip15pt - \frac{c}{\mu(t)\mu_1(t)}\,\int u\,(q(\partial_x) \,[q(\partial_x); \phi_{\sigma}\Big(\frac{x}{\mu_1(t)} \Big)]u)\,dx\\
&\hskip15pt- \frac{c}{\mu(t)}\,\int u R_1(\cdot)u\,dx\\
&= B_{1,1}(t)+B_{1,2}(t)+B_{1,3}(t).
\end{split}
\end{equation}

We observe that $B_{1,1}(t)$ is the term we want to estimate.

To estimate $B_{1,3}(t)$ we use the Cauchy-Schwarz inequality, the estimate \eqref{ilw-e16} in Lemma \ref{est-1}, the
conservation of mass and the definitions of $\mu(t)$ and $\mu_1(t)$ to yield
\begin{equation}\label{est-b13}
\begin{split}
|B_{1,3}(t)| &\le \frac{c}{\mu(t) \mu_1(t)^{\frac52}}\|\phi''\|^{1/2}_2\|\phi'''\|_2^{1/2}\|u_0\|_2^2\\
&\lesssim_{\sigma} \frac{(\log t)^{1/2}}{t^{1+\frac32b}} \in L^1(\{t\gg1\}).
\end{split}
\end{equation}

Employing the Cauchy-Schwarz inequality, Lemma \ref{est-3}, the conservation of mass, and the properties of $\mu(t)$ and $\mu_1(t)$ 
we obtain
\begin{equation}\label{est-b12}
\begin{split}
|B_{1,2}(t)| &\le \frac{c}{\mu(t)\mu_1(t)^2\mu_1(t)^{1/2}} \|u_0\|_{2}^2\|\phi_{\sigma}''\|_2^{1/2}
\Big(\|\phi_{\sigma}'\|_2^{1/2}+\|\phi_{\sigma}'''\|_2^{1/2})\\
&\lesssim_{\sigma} \frac{(\log t)^{1/2}}{t^{1+\frac32b}} \in L^1(\{t\gg1\}).
\end{split}
\end{equation}

\vspace{.5cm}

Finally, notice that by Lemma \ref{GNSinequality}
    \begin{equation}\label{l61-12}
    \begin{aligned}
    & \int_{\mathbb{R}}|u|^3 \phi_{\sigma}\Big(\frac{x}{\mu_1(t)} \Big)dx  \\
    &  \leq \sum_{n \in \mathbb{Z}}\int_{\mathbb{R}}(|u|\,\zeta_n)^3 \phi_{\sigma}\Big(\frac{x}{\mu_1(t)} \Big)dx  \\
    & \leq \sum_{n \in \mathbb{Z}}\|u\, \zeta_n\|^3_{L^3}\Big(\sup_{x \in [n, n+ 1]} \phi_{\sigma}\Big(\frac{x}{\mu_1(t)} \Big)\Big)  \\
    & \lesssim \sum_{n \in \mathbb{Z}}\|u\,\zeta_n\|^2_{L^2}\|D^{1/2}_x(u\,\zeta_n)\|_{L^2}\Big(\sup_{x \in [n, n+ 1]} \phi_{\sigma}\Big(\frac{x}{\mu_1(t)} \Big)\Big).  
    \end{aligned}
    \end{equation}
   
Moreover, by Lemma \ref{lem1} and hypothesis,
    \begin{equation}\label{l61-13}
    \begin{aligned}
     \|D^{1/2}_x(u\zeta_n)\|_{L^2}& \lesssim \left\|D^{1/2}_xu(t)\right\|_{L^2}\|\zeta_n\|_{L^{\infty}} + \|u(t)\|_{L^2}\|D^{1/2}_x\zeta_n\|_{L^{\infty}} \\
     & \lesssim \|u(t)\|_{H^{1/2}(\mathbb{R})} \lesssim \|u\|_{L_t^{\infty}H^{1/2}}.
     \end{aligned}
     \end{equation}
  
Combining these estimates we deduce that
\begin{equation*}\label{l61-14}
 \int_{\mathbb{R}}|u(x,t)|^3 \phi_{\sigma}\Big(\frac{x}{\mu_1(t)} \Big)dx\\
  \lesssim  \sum_{n \in \mathbb{Z}}\|u\,\zeta_n\|^2_{L^2}\Big(\sup_{x \in [n, n+ 1]} \phi_{\sigma}\Big(\frac{x}{\mu_1(t)} \Big)\Big).
\end{equation*}
A similar analysis to that given in Lemma 4.1 in  \cite{MMPP} (see also \cite{KM}) yields
\begin{equation}\label{l61-15}
 \int_{\mathbb{R}}|u(x,t)|^3 \phi_{\sigma}\Big(\frac{x}{\mu_1(t)} \Big)dx \lesssim \int_{\mathbb{R}}|u(x,t)|^2 \phi_{\sigma}\Big(\frac{x}{\mu_1(t)} \Big)\,dx .
\end{equation}

Using the properties of the function $\phi$ in \eqref{weight} for suitable $\lambda$ and $\sigma$  we can apply Lemma \ref{intl2bounded} to deduce
that $B_3(t) \in L^1(\{t \gg 1\})$. 

\medskip

Collection the information in \eqref{l61-4}, \eqref{l61-5}, \eqref{est-b13}, \eqref{est-b12} and \eqref{l61-15}
we deduce that
\begin{equation*}
\frac{1}{t\log t}\int_{B_{t^b}} |q(\partial_x)u(x,t)|^2\,dxdt\le \frac{d}{dt}\mathcal{J}(t)+g(t),
\end{equation*}
where  $\mathcal{J}(t)$ is bounded and $g(t)\in L^1(\{t\gg 1\})$.

A similar analysis as the one implemented in the proof of Theorem \ref{L2ILW} yields the desired result.  
    
\end{proof}

The reminder of the proof of Theorem \ref{energy-ILW} uses a similar argument as the proof of Theorem \ref{L2ILW}, so we will omit it.

\vspace{1cm}

\subsection{Sketch Proof of Corollary \ref{corL2ILW}}

\begin{proof}[Sketch Proof of Corollary \ref{corL2ILW}]

The proof of this result follows closely the argument employed to establish Theorem \ref{L2ILW}. Thus we will
describe only the new elements in the proof. 

We consider the functional 
\begin{equation*}
\mathcal{I}_{\rho}(t)=\frac{1}{\mu(t)}\int u(x,t)\,\psi_{\sigma}\left(\frac{x-\rho(t)}{\mu_1(t)}\right)\phi_{\lambda}\left(\frac{x-\rho(t)}{\mu_1^q(t)}\right)\,dx
\end{equation*}
where $\rho(t)=\pm t^m$, $m$ as in the statement of the corollary, $\mu(t)$ and $\mu_1(t)$ defined as in \eqref{parameters}, and
$\psi_{\sigma}$ and $\phi_{\lambda}$ defined as above.

\medskip 

Same analysis as in the proof of Lemma \ref{Functionalbounded} yields
\begin{equation*}
\sup_{t\gg 1} | \mathcal{I}_{\rho}(t)| <\infty.
\end{equation*}

We then get a similar inequality as \eqref{NC-2} in Lemma \ref{BoundL1}, that is,
\begin{equation*}
\begin{aligned}
& \frac{1}{\mu_1(t) \mu(t)}\int u^2(x,t)\,\psi_{\sigma}'\left(\frac{x-\rho(t)}{\mu_1(t)}\right)\phi_{\sigma}\left(\frac{x-\rho(t)}{\mu_1^q(t)}\right)\,dx
\\
&\qquad  \le 4\frac{d}{dt}\mathcal{I}_{\rho}(t)+h_{\rho}(t),
\end{aligned} 
\end{equation*}
where $h_{\rho}(t)\in L^1(\{t\gg1\})$. In addition to the terms previously handle in the proof of Lemma  \ref{BoundL1}, we must estimate two new extra terms, 
\[
- \frac{\rho'(t)}{\mu_1(t)\mu(t)}\int_{\mathbb{R}} u(x,t) \psi_{\sigma}'\left( \frac{x-\rho(t)}{\mu_1(t)} \right) \phi_{\lambda}\left( \frac{x-\rho(t)}{\mu_1^q(t)} \right)\,dx=A(t)
\]
and 
\[
- \frac{\rho'(t)}{\mu_1^q(t)\mu(t)}\int_{\mathbb{R}} u(x,t) \psi_{\sigma}\left( \frac{x-\rho(t)}{\mu_1(t)} \right) \phi_{\lambda}'\left( \frac{x-\rho(t)}{\mu_1^q(t)} \right)\,dx=B(t).
\]
This can be done by using the Cauchy-Schwarz inequality and the mass conservation to lead to
\begin{equation*}
\begin{split}
|A(t)+B(t)| &\le \left|\frac{\rho'(t)}{\mu_1^{1/2}(t)\mu(t)} \right| \|u_0\|_{L^2}\|\psi_{\sigma}'\|_{L^2}  \|\phi_{\lambda}\|_{L^\infty}\\
&\hskip10pt + \left|\frac{\rho'(t)}{\mu_1^q(t)\mu(t)}\right| \|u_0\|_{L^2} \|\psi_{\sigma}\|_{L^{\infty}}  \|\phi_{\lambda}'\|_{L^2}\\
& \lesssim_{\sigma,\lambda, m} \frac{1}{t^{(4-2m -b)/2 }\log^{3/2}t}+\frac{1}{t^{(4 -2b -2m + bq)/2}\log^{(4-q)/2}t}.
\end{split}
\end{equation*}

We note that the last terms are in $L^1(\{t\gg 1\})$ since for the first term $m \leq 1 - \frac{b}{2}$ and for the second one $m \leq 1 - \frac{b}{2} < 1 - \frac{b(1 - q)}{2}$, respectively.

\medskip

From this point on the utilised argument to establish Theorem \ref{L2ILW} can be applied to complete the proof of Corollary \ref{corL2ILW}.

\end{proof}

\vspace{0.5cm}
\noindent{\bf Acknowledgements.} The authors are grateful to Argenis Mendez for his valuable comments and suggestions which improve the presentation of this work. F.L. was partially supported by CNPq grant 310329/2023-0 and FAPERJ grant E-26/200.465/2023. 
\medskip

\section*{conflict of interest statement}

On behalf of all authors, the corresponding author states that there is no conflict of interest. 

\section*{Data availability}

We do not analyze or generate any datasets, because our work proceeds within a theoretical and mathematical approach.

\end{document}